\documentclass[11pt]{article}
\usepackage{amsmath,amsthm,amssymb,bbm}
\usepackage{tikz}
\usepackage{mathrsfs}
\usepackage{dsfont}

\usepackage{todonotes}
\usepackage{hyperref}
\usepackage{xcolor}
\hypersetup{
	colorlinks,
	linkcolor={red!50!black},
	citecolor={blue!50!black},
	urlcolor={blue!80!black}
}

\usepackage{epsfig}

\theoremstyle{theorem}
\newtheorem{thm}{Theorem}[section]
\newtheorem{cor}[thm]{Corollary}
\newtheorem{lem}[thm]{Lemma}
\newtheorem{prop}[thm]{Proposition}

\theoremstyle{definition}

\newtheorem{rem}[thm]{Remark}

\numberwithin{equation}{section}
\newcommand{\1}{\mathds{1}}
\newcommand{\bigabs}[1]{\big|#1\big|}
\newcommand{\biggabs}[1]{\bigg|#1\bigg|}
\newcommand{\abs}[1]{|#1|}
\newcommand{\T}[1]{{#1}^{\mathrm{T}}}
\newcommand{\scalar}[1]{\langle{#1}\rangle}
\newcommand{\bigscalar}[1]{\big\langle{#1}\big\rangle}
\newcommand{\norm}[1]{\lVert{#1}\rVert}
\newcommand{\bignorm}[1]{\bigg\lVert{#1}\bigg\rVert}
\newcommand{\Id}{\mathrm{Id}}
\newcommand{\transpose}[1]{#1^{\top}}

\newcommand{\de}{\overset{d}{=}}
\newcommand{\dc}{\overset{d}{\rightarrow}}
\newcommand{\Rnot}{\mathbb{R}\setminus\{0\}}
\newcommand{\Rplus}{\mathbb{R}_{\ge}}
\newcommand{\tr}{\mathbf{tr}}
\newcommand{\SpaceRO}{\R_{\ge}\times\OrthogonalGroup}
\newcommand{\vecc}{\mathbf{vec}}
\newcommand{\Cov}{\mathrm{Cov}}

\def\C{{\mathbb{C}}}
\def\d{{\text{d}}}
\def\Q{{\widehat{Q}_+}}
\def\P{{\mathbb{P}}}
\def\R{{\mathbb{R}}}
\def\N{{\mathbb{N}}}

\def\E{{\mathbb{E}}}
\def\V{{\mathbb{V}}}
\def\Support{\mathbb{S}}
\def\T{{\mathbf{T}}}
\def\OrthogonalGroup{{\mathbb{O}(3)}}
\def\U{{\mathbb{U}}}
\def\G{\mathbb{G}}

\def\I{{\mathrm{I}}}

\def\mZ{{\mathcal{Z}}}
\def\mF{{\mathcal{F}}}
\def\mW{{\mathcal{W}}}

\def\bB{\mathbf{B}}
\def\as{\mathrm{a.s.}}
\def\sL{{\mathrm{L}}}
\def\sU{{\mathrm{U}}}

\renewcommand{\le}{\leqslant}\renewcommand{\leq}{\leqslant}
\renewcommand{\ge}{\geqslant}
\renewcommand{\setminus}{\smallsetminus}

\allowdisplaybreaks

\newcommand{\Sim}{\mathds{M}}

\renewcommand{\vec}[1]{\mathbf{vec}(#1)}

\begin{document}
	
	\title{A probabilistic study of the set of stationary solutions to  spatial kinetic-type equations}

\author{Sebastian Mentemeier and Glib Verovkin\thanks{Universit\"at Hildesheim, Institut f\"ur Mathematik und Angewandte Informatik, Hildesheim 31141, Germany. E-Mail: mentemeier@uni-hildesheim.de; verovkin@uni-hildesheim.de \\
The authors were partially supported by DFG grant ME 4473/1-1. }}

\maketitle

\begin{abstract}
	 
	 In this paper we study multivariate kinetic-type equations in a general setup, which includes in particular the spatially homogeneous Boltzmann equation with Maxwellian molecules, both with elastic and inelastic collisions. Using a representation of the collision operator derived in \cite{Bassetti+Ladelli+Matthes:2015,Dolera2014},
		  we prove the existence and uniqueness of time-dependent solutions with the help of continuous-time branching random walks, under assumptions as weak as possible. Our main objective is a characterisation of the set of stationary solutions, e.g. equilibrium solutions for inelastic kinetic-type equations, which we describe as mixtures of multidimensional stable laws. \\[.2cm]
\end{abstract}

\section{Introduction} 

\subsection{Kinetic-type equations}

In this article, we study kinetic-type equations that originate from the study of the time-dependent distribution of particle velocities in a spatially homogeneous, dilute ideal gas. Particles are assumed to interact only via binary collisions and to move independently before and after the collision has taken place ({\em Boltzmann's Stosszahlansatz}). 
There is an enormous body of literature devoted to the study of these equations, see \cite{Villani:2002} for a survey and \cite{Cercignani+Reinhard+Pulvirenti:2012} for a textbook introduction. 
If $f(v,t)$ denotes the density of particles with velocity $v \in \R^3$ at time $t \ge 0$, the spatial homogeneous Boltzmann equation reads
\begin{align} \label{eq:Boltzmann}
	\frac{\partial}{\partial t} f &= Q(f,f) \nonumber \\ 
	f(\cdot,0) &= f_0. 
\end{align}
where $Q$ is the so-called {\em collision operator}, describing the changes in the velocity distribution due to the collison of two random particles.
This equation should be understood in a formal way. In general, one does not assume that the distribution of particle velocities is absolutely continuous.

The form of the collision operator $Q$ depends on the type of the interaction between particles. The case of {\em Maxwellian molecules}, which interact with a repulsive force that decays like $r^{-5}$ has received a lot of attention, for it is mathematically tractable and yet provides good predictions (see \cite[Chapter 1.1.]{Villani:2002} for a detailled discussion).

In the setting of Maxwellian molecules, it makes sense to consider the Boltzmann equation in Fourier space, for there is an explicit expression for the ``Fourier version'' $\widehat{Q}$ of $Q$. In particular, $\widehat{Q}$ splits into a {\em collisional gain operator} $\widehat{Q}_+$ and a loss operator, which in this case is just the identity map. If $\phi_t$ denotes the Fourier transform of the particle velocity distribution $f(\cdot,t)$ at time t, then Eq.~\eqref{eq:Boltzmann} becomes
\begin{align} \label{eq:BoltzmannFourier}
	\frac{\partial}{\partial t} \phi_t = \widehat{Q}_+(\phi_t, \phi_t) - \phi_t , \qquad \phi_0=\phi_{f_0}.
\end{align}

Under the {\em cut-off assumption} that excludes the effect of grazing collisions 
(which would lead to a singularity in the collision kernel associated to $\widehat{Q}_+$), Bassetti, Ladelli and Matthes in ~\cite{Bassetti+Ladelli+Matthes:2015}, following Dolera and Regazzini \cite{Dolera2014}, have derived a probabilistic representation of $\widehat{Q}_+$, which is as follows. Denote by $SO(d)$ the special orthogonal group in dimension $d$, write $\top$ for transpose and denote by $e_i$, $i \in \{1,2,3\}$ the standard orthonormal basis of $\R^3$.  Let $\phi$ be the Fourier transform of a probability measure on $\R^3$, then for all $r \ge 0$ and $o \in SO(3)$,
\begin{equation} \label{eq:Maxwell}  \widehat{Q}_+\bigg(\phi, \phi \bigg) (r o e_3) = \E \bigg[ \phi \bigg(r  o R_1 O_1 e_3\bigg) \phi\bigg( r o R_2 O_2 e_3 \bigg) \bigg], \end{equation} 
where $R_1, R_2$ are nonnegative  random variables and $O_1$, $O_2$ are random rotations in $SO(3)$ having suitable probability distributions, see \cite[Section 3]{Bassetti+Ladelli+Matthes:2015} for details.

\medskip

In this paper, we are going to study the initial value problem associated to Eq. \eqref{eq:BoltzmannFourier} and the operator \eqref{eq:Maxwell} under assumptions as mild as possible.
 We are going to provide a representation of time-dependent solutions in terms of a continuous-time multitype branching process (see Theorem \ref{K:time-dependent_solution}). In our main result,  Theorem \ref{K:proposition_main_result}, we determine the set of solutions to the associated stationary equation
\begin{equation}\label{K:stationary_equation}
	\phi=\widehat{Q}_+(\phi,\phi),
\end{equation}
the solutions of which govern the long-term behavior of the time-dependent solutions. We consider both the elastic case (conservation of energy) as well as the inelastic case (loss of energy).

\subsection{Setup and assumptions}

We are going to study both time-dependent and stationary solutions to Eq. \eqref{eq:BoltzmannFourier} and \eqref{K:stationary_equation} respectively, assuming that $r, R_1, R_2$ take nonnegative  values and that $o,O_1,O_2$ belong to the group of orthogonal $3 \times 3$ matrices $\OrthogonalGroup$. In this section we present our assumptions imposed on the operator \eqref{eq:Maxwell} and the laws of $R_1,R_2,O_1,O_2$, which comfortably cover the case of the (inelastic) Boltzmann equation with Maxwellian molecules, as proved in \cite[Section 3]{Bassetti+Ladelli+Matthes:2015}.

Considering the function $m: [0, \infty) \to [0,\infty]$  defined by
\begin{equation*}
	m(\gamma):=\E[R_1^{\gamma}+R_2^{\gamma}],
\end{equation*}
our first assumption is
\begin{equation}\label{K:assumption_m} \tag{A1}
	m(\alpha) = 1\quad \text{for some }\alpha\in(0,2]
\end{equation}
The physical interpretation is as follows: $\alpha=2$ corresponds to the conservation of kinetic energy, while $0<\alpha<2$ means dissipation of energy, see e.g. \cite[Section 4]{Pulvirenti2004}.
One can interprete $R_1, R_2$ as the proportions of kinetic energy given to each one of two colliding particles. In this application it then follows $R_1 \le 1$ and $R_2\le 1$ a.s., with $\P(R_1=1 \text{ or } R_2=1)<1$. Hence,  the existence of $\alpha>0$ as in \eqref{K:assumption_m} is always guaranteed --
observe that the function $m$ is convex and hence continuous and differentiable on the interior of its domain. Also our second assumption,
\begin{equation}\label{K:assumption_derivative_m} \tag{A2}
	m'(\alpha)\in(-\infty,0)
\end{equation}
for the $\alpha$ from \eqref{K:assumption_m}, is always satisfied in the setting of kinetic equations.

 In order to make the definition of $\widehat{Q}_+$ in \eqref{eq:Maxwell} independent of a particular choice of $o$, we need to assume that
\begin{equation}\label{K:assumption_uniqueness} \tag{A3}
	(oR_1O_1e^{}_3,oR_2O_2e^{}_3)\overset{d}{=}(uR_1O_1e^{}_3,uR_2O_2e^{}_3)
\end{equation}
for any $o,u\in\OrthogonalGroup$, such that $oe^{}_3=ue^{}_3$. Here and below, $\stackrel{d}{=}$ denotes equality in distribution. 
Assumption \eqref{K:assumption_uniqueness} may  look restrictive at a first glance, but it is required for $\widehat{Q}_+$ to be well defined and it is satisfied in the main applications from statistical mechanics, see e.g.\ \cite[Section 3]{Bassetti+Ladelli+Matthes:2015} .
A trivial example, where \eqref{K:assumption_uniqueness} is satisfied, is when both $O_1,O_2$ act on the plane, spanned by $\{e_1,e_2\}$, i.e are of the form
\begin{equation*}
	\begin{pmatrix}
		\star & \star & 0 \\
		\star & \star & 0 \\
		0 & 0 & 1
	\end{pmatrix}.
\end{equation*}
In this case, moreover, the equality in \eqref{K:assumption_uniqueness} holds almost surely.
Note that, except for \eqref{K:assumption_uniqueness}, we allow for any dependence structure  within and between the tuples $(R_1,O_1)$ and $(R_2,O_2)$.

When studying stationary solutions satisfying Eq. \eqref{K:stationary_equation}, we assume that $\phi$ is the characteristic function of a random vector $X$ with the property
\begin{equation} \tag{A4}\label{K:assumption_regular_variation}
	\P(\norm{X}>t)\le Ct^{-\alpha}\quad  \text{for all } t>0
	\end{equation}
	for some constant $C$.
This assumption is much weaker than assuming that $\norm{X}$ is in the normal domain of attraction of an $\alpha$-stable law, which would entail that $\P(\norm{X}>t)$ is of precise order $t^{-\alpha}$, see \cite[§ 35]{Gnedenko1954}.

\subsection{Adhoc statement of the main result}

To not repeat ourselves, we will introduce all necessary notation and objects in Section \ref{sect:branching processes}  before stating the main result in full detail in Section \ref{sect:main results}. However, for the reader's convenience, we formulate an adhoc version of the classification result for stationary solutions in the most simple, yet instructive case, when $\alpha$ in \eqref{K:assumption_m} is in $(0,1)$.

Denote by $\OrthogonalGroup$ the multiplicative group of orthogonal 3$\times3$ matrices, by $\R_>$ ($\R_\ge$) the multiplicative group of positive (the set of nonnegative) real numbers, and denote by $\Sim:=\R_>\times\OrthogonalGroup$ the group of similarity matrices, \textit{i.e.}, products of an orthogonal matrix and a dilation. Equipping each set with the corresponding Borel $\sigma$-field, we obtain measurable spaces.
We denote by $\Support$ the smallest closed subgroup of $\Sim$  that contains $\{(R_1,O_1), (R_2,O_2)\}$ with probability 1.

\begin{thm}
	Assume \eqref{K:assumption_m}-\eqref{K:assumption_regular_variation} with $\alpha \in (0,1)$. There exists a nonnegative random variable $W_\infty$ with mean one such that a characteristic function $\phi$ is a solution to \eqref{K:stationary_equation} if and only if
	$$ \phi(roe_3)=\E[e^{-W_\infty K(r,o)}]$$
	for a function $K: \SpaceRO \to \C$ satisfying $K(r,o)=s^{-\alpha} K(rs,ou)$ for all $(s,u)\in \Support$ and $K(0,o)=0$ for any $o\in\OrthogonalGroup$. 
\end{thm}

\subsection{Related works}
The probabilistic form of the Boltzmann equation that is used in our work is based on \cite{Bassetti+Ladelli+Matthes:2015}, whereas the original idea behind it has been first introduced in\linebreak \cite{Dolera2014}.
In \cite{Dolera2014} time-dependent solutions to the spatially homogeneous Boltzmann equation have been considered assuming existence of the fourth moment of the initial condition $f_0$. The authors derived a uniform upper bound on the total variation distance between the solution and the limiting Maxwellian distribution. An explicit form of the stationary solution has also been established, given by a weak limit of its time-dependent counterpart.
In \cite[Theorem 2.3]{Bassetti+Ladelli+Matthes:2015} the set of stationary solutions in the domain of normal attraction of an $\alpha$-stable law was studied, assuming in addition the existence of a density. These two assumptions will in particular be removed in our main result, Theorem \ref{K:proposition_main_result}.
We also refer the reader to a list of further references given in that paper, see \cite[Section 1.2]{Bassetti+Ladelli+Matthes:2015}. 

Starting from the seminal work of Kac \cite{Kac1956}, various probabilistic interpretations of solutions have been investigated, we mention here in particular \cite{Carlen2000, Gabetta2008, Bassetti+Ladelli+Matthes:2010, Bassetti2012}. In our work, we will introduce for the multivariate setting a representation of solutions with the help of continuous-time branching random walks, based on the univariate techniques developed recently in \cite{Bogus+Buraczewski+Marynych:2020, Buraczewski+Kolesko+Meiners:2021,Buraczewski+Dyszewski+Marynych:2023}.

 To study stationary equations, we pursue strategies that have been developed for the study of fixed points of {\em smoothing equations}, see e.g. \cite{Alsmeyer+Biggins+Meiners:2012, Alsmeyer+Meiners:2013(9),Meiners+Mentemeier:2017}. 
We may refer to $\widehat{Q}_+$ as nonlinear smoothing operator because of the following relation. A random vector $X$ is said to satisfy a (linear) multivariate smoothing equation with matricial weights, if 
\begin{equation}\label{eq:homogeneous smoothing}
	X \stackrel{d}{=} \sum_{j=1}^N T_j X_j,
\end{equation}
where  $X_1, X_2, \dots$ are i.i.d.\ copies of $X$, that are independent of the (conditionally given) weights $T_1, \dots, T_N \in \R^{d \times d}$; here the number $N$ of weights may also be random. Then the characteristic function $\phi$ of $X$ satisfies the equation
\begin{equation}\label{eq:FE homogeneous smoothing}
	 \phi(\xi) = \E \bigg[ \prod_{j=1}^N \phi(T_j^\top \xi) \bigg]
\end{equation} for all $\xi \in \R^d$.
Despite the formal similarity of equations \eqref{eq:FE homogeneous smoothing} and \eqref{eq:Maxwell}, the action of $\widehat{Q}_+$ does not allow for a representation as a weighted sum of random variables as in \eqref{eq:homogeneous smoothing}, due to the noncommutativity of matrix products: $o$ and $O_1$, $O_2$ do not commute.

The set of solutions to \eqref{eq:homogeneous smoothing} in the case where $T_j \in \Sim$, has been determined in \cite{Meiners+Mentemeier:2017}. It is instructive to compare it with our result.

\subsubsection*{Structure of the paper}
We proceed by introducing necessary notation and objects from the realm of branching processes in Section \ref{sect:branching processes}. With this at hand, we are ready to formulate our results in full detail in Section \ref{sect:main results}. Proofs for the existence and uniqueness of time-dependent solutions are given in Section \ref{sect:proofs time dependent}; the study of stationary solutions is conducted in Section \ref{sect:proofs stationary}. In Section \ref{sect:geometry} we investigate the structure of stationary solutions in more detail.

\section{Branching processes and martingales} \label{sect:branching processes}

A  probabilistic representation (using branching processes) of time-dependent solutions to Eq. \eqref{eq:BoltzmannFourier} was given in \cite{Bassetti+Ladelli+Matthes:2010}, where the authors used a discrete-time model based on Wild series to derive a time-dependent solution.
In \cite{Bogus+Buraczewski+Marynych:2020,Buraczewski+Dyszewski+Marynych:2023}
it was observed that this construction can be simplified by using a continuous-time branching random walk. We will follow their construction to obtain a representation of both time-dependent and stationary solutions as functionals of a certain branching process.

\subsection{A weighted branching process}
Let $\T=\bigcup_{n\in\N_0}\{1,2\}^n$ be an infinite binary tree with the root $\emptyset$. For a node $v=v_1\ldots v_n\in\T$, where $v_i\in\{1,2\}$ for any $i=1,\ldots,n$, we say that $v$ is in the $n$-th generation and write $\abs{v}=n$. For any $k=0,\ldots,n-1$ the ancestor of $v$ in the $k$-th generation is denoted by $v|_k:=v_1 \cdots v_k$. For any $j\in\{1,2\}$ the $j$-th descendant of $v$ in the $(n+1)$-th generation is  $vj:=v_1\ldots v_nj$. 

Let $E$ be an exponential random variable with unit mean, independent of $(R_1,O_1,R_2,O_2)$. Assign to each node $v \in \T$ an i.i.d. copy $(E(v),R_1(v),O_1(v),R_2(v),O_2(v))$ of the tuple $(E,R_1,O_1,R_2,O_2)$.
 We call the family $\bB:=(E(v),R_j(v),O_j(v):j\in\{1,2\})_{v\in\T}$ the weighted branching process associated to the tuple $(E,R_j,O_j:j\in\{1,2\})$, rooted at $\emptyset$. 
We denote by $F:=(F_n)_{n \in \N_0}$ the natural filtration for $\bB$, associated to generations, namely
\begin{equation}\label{K:discrete_time_sigma_algebra}
	F_n:=\sigma((E(v),L(v),U(v)),\abs{v}\leq n)
\end{equation} 
for any $n\in\N_0$. 

For any $w\in\T$ we consider the shifted weighted branching process rooted at $w$, defined by  $$[\bB]_w:=(E(wv),R_j(wv),O_j(wv):j\in\{1,2\})_{v\in\T}.$$
Notice, that  $\bB$ and $[\bB]_w$ have the same law. For any function of the weighted branching process $\zeta=\zeta(\bB)$ we can define a function $[\zeta]_w$ as $\zeta([\bB]_w)$. 

For each node $v\in\T$ we define 
 recursively the quantities
\begin{equation}
	L(\emptyset):=1,\quad L(vj):=L(v)R_j(v)
\end{equation}
and 
\begin{equation}
	U(\emptyset):=\Id, \quad U(vj):=U(v)O_j(v),
\end{equation}
with $\Id \in \OrthogonalGroup$ denoting the identity matrix.
Therefore, for any $v=v_1\ldots v_n$ 
we have
\begin{equation} 
	L(v):=R_{v_1}(\emptyset)\ldots R_{v_n}(v|_{n-1})
\end{equation}
and
\begin{equation} 
	U(v):=O_{v_1}(\emptyset)\ldots O_{v_n}(v|_{n-1}).
\end{equation}
Observe that, by \cite[Lemma 7.2]{Liu1998}, assumption \eqref{K:assumption_m}  implies 
\begin{equation}\label{K:supremum_L(v)_goes_to_0}
	\lim_{n\rightarrow\infty}\sup_{\abs{v}=n}L(v)=0\quad\as
\end{equation}

Recall that $\Support$ denotes the smallest closed subgroup of $\Sim$  that contains $\{(R_1,O_1), (R_2,O_2)\}$ with probability 1. Hence $\{(L(v),U(v))\, : \, {v\in\T}\} \subset \Support$ with probability 1 as well.
In the same way, let $\G$ and $\U$ be the smallest closed subgroups of $\R_>$ and $\OrthogonalGroup$,  such that $\P\big(\{L(v) \, : \, v \in \T\} \in \G\big)=1$ and $\P \big(\{U(v) \, : \, v \in \T\} \in \U\big)=1$, respectively. 
Note, that in general $\G\times \U\neq \Support$.

\subsection{Branching random walk}
In what follows we will consider nodes in $\T$ as particles of a multiplicative  continuous-time branching random walk on the group $\Sim=\R_>\times\OrthogonalGroup$. For any  such $v\in\T$ the component $E(v)$ will be interpreted as the lifetime of a particle $v$, whereas the tuple $(L(v),U(v))$ will describe the position of $v$ in space $\SpaceRO$. However, before we are able to introduce the formal definition of the branching random walk, certain technical steps should be done first.

For each $v \in \T$ we define the following two random variables
\begin{equation}\label{K:birth_death_times}
	B(v):=\sum_{k=1}^{\abs{v}-1}E(v|_k)\quad\text{and}\quad D(v):=B(v)+E(v).
\end{equation} 
denoting the birth time and the death time of particle $v$, respectively, with the convention $B(\emptyset)=0$.
Further, we define for any $t\ge 0$ the following two random subsets of $\T$:
\begin{equation}
	\T_t:=\{v\in\T:B(v)\leq t\}
\end{equation} 
and
\begin{equation}
	\partial \T_t:=\{v\in\T:B(v)\leq t< D(v)\}.
\end{equation} 
The first set contains all particles which have been born up to time $t$, whereas the second one contains particles that are alive at time $t$. We denote by $\mF:=(\mF_t)_{t \ge 0}$ the filtration  generated by particles in the set $\T_t$,
\begin{equation}\label{continuous_time_sigma_algebra}
	\mF_t:=\sigma(E(v),L(v),U(v):v\in\T_t).
\end{equation}
It can easily be checked, that 
\begin{equation*}
	\mF_0=F_0\quad\text{and}\quad\lim_{t\rightarrow\infty}\mF_t=\lim_{n\rightarrow\infty}F_n.
\end{equation*}

We will denote the above limit $\lim_{n\rightarrow\infty}F_n=:F_\infty$, which can informally be  perceived as the total information conveyed by the process $\bB$.

A continuous-time multiplicative branching random walk in $\Sim$ is then given by the point process $\mZ=(\mZ_t)_{t\in\Rplus}$, defined by
\begin{equation}\label{K:continuous_BRW}
	\int f(y,u)\d\mZ_t(y,u) := \sum_{v\in\partial T_t} f(L(v),U(v))
\end{equation}
for all $t \ge 0$ and nonnegative measurable functions $f$ on $\Sim$. 
Its dynamics can be described as follows. An initial particle is located in $(1,\Id)$. After an exponential time with unit mean it dies and gives birth to two new particles, which are distributed on $\Sim$ according to the point process $\sum_{j=1}^{2}\delta_{(R_j,O_j)}$. The two descendants reproduce in exactly the same way as their ancestor, independently from the latter and also from each other.

As in formula \eqref{K:continuous_BRW}, throughout our work we will omit the lower index under the integral sign whenever the integration area is the entire domain of an underlying measure.

Since $(\mZ_t)_{t\in\Rplus}$ is, by construction, a function of a weighted branching process $\bB$, a shifted continuous-time branching random walk $([\mZ_t]_w)_{t\in\Rplus}$ is well-defined for any $w\in\T$:
\begin{equation*}
	[\mZ_t]_w:=\sum_{v\in[\partial \T_t]_w}\delta_{(L(wv),U(wv))},\quad t\ge 0,
\end{equation*}
where $[\partial \T_t]_w=\{v\in\T:B(wv)\leq t< D(wv)\}$.

By sorting all particles alive at time $t+h$ according to their ancestors alive at time $t$, we derive the following branching property for $\mZ$,  which will be given without proof.

\begin{lem}\label{K:proposition_continuous_BRW_branching_property}
	For all $t,h \ge 0$ and all measurable sets $A \in \R_>$, $B \in \OrthogonalGroup$ it holds
	\begin{equation}\label{K:continuous_BRW_branching_property}
		\mZ_{t+h}(A,B)\de\sum_{w\in\partial \T_t}[\mZ_{h}]_w\bigg(L(w)^{-1}A,U(w)^{-1}B\bigg).
	\end{equation}
\end{lem}

\subsection{Associated random walks}
In this section, we define two associated random walks, one in continuous time and one in discrete time, by means of many-to-one formulae.

\begin{lem}\label{K:many-to-one-lemma}
	Consider $\gamma\ge 0$ with $m(\gamma)<\infty$. Let $(\sL^{(\gamma)}_n,\sU^{(\gamma)}_n)_{n\in\N_0}$ be a multiplicative random walk on $\Sim$ starting in  $(\sL^{(\gamma)}_0,\sU^{(\gamma)}_0):=(1,\Id)$ and increment
	law $\rho$ defined by
	$$ \rho(A \times B) := \frac{1}{m(\gamma)}\E \bigg[\sum_{j=1}^2R_j^{\gamma} \,  \1_A(R_j) \1_B(O_j) \bigg]$$
	for all measurable $A \subset \R_>$, $B \subset \OrthogonalGroup$.
 Then, for any $n \in \N$ and any nonnegative measurable function $h$ the following identity holds:
	\begin{align}
		& \E \big[ h(\sL^{(\gamma)}_1,\ldots,\sL^{(\gamma)}_n,\sU^{(\gamma)}_1,\ldots,\sU^{(\gamma)}_n) \big] \nonumber \\
		&=\frac{1}{m(\gamma)^n}\E \bigg[ \sum_{\abs{v}=n} L(v)^{\gamma} h\bigg((L(v|_k))_k, (U(v|_k))_k;~k=1,\ldots,n\bigg) \bigg] \label{K:many-to-one_argument}
	\end{align}
\end{lem}

Many-to-one results are standard tools in the theory of branching processes. We refrain from giving details of the proof. A proof for a univariate version can be found e.g. in \cite[Theorem 1.1]{Shi:2012}.

 We will be particularly interested in the case $\gamma=\alpha$ and, to simplify the notation, we will write $\sL_n:=\sL^{(\alpha)}_n$ and $\sU_n:=\sU^{(\alpha)}_n$, for any $n\ge 0$.
We observe that $\Support$ is also the smallest closed semigroup containing $(\sL_n,\sU_n)_{n \ge 0}$ with probability 1, as well as $\U$ and $\G$ contain $(\sU_n)_{n \ge 0}$ and $(\sL_n)_{n \ge 0}$ with probability 1, respectively. 
We finish this section by quoting the following Choquet-Deny-type result, which proves to be useful for studying solutions to the stationary equation.
\begin{lem}\label{K:Choquet-Deny-lemma} 
	Let $\rho$ be a probability law on $\Sim$ and let $\Support$ be the smallest closed subgroup of $\Sim$, generated by $\rho$. Suppose $\xi:\Support\rightarrow\R$ is measurable and bounded. Then, if
	\begin{equation}\label{K:Choquet-Deny_functional_equation}
		\E \bigg[\xi(s\sL_1,u\sU_1)\bigg] = \xi(s,u)
	\end{equation}
	holds for any $(s,u)\in\Support$, where the expectation is taken w.r.t. to $\rho$, then $\xi$ is constant $\rho$-a.e.
\end{lem}

\begin{proof}[Source]
		A Choquet-Deny-type result for the similarity group $\Sim=\R_> \times \OrthogonalGroup$ is a consequence of \cite[Theorem 3]{Guivarch1973}. We refer the reader to \cite[Lemma 5.1]{Meiners+Mentemeier:2017} for a version with notation adopted to the present context, mutatis mutandis: In the quoted lemma, new information enters from the left, while we consider multiplication from the right.
\end{proof}

\subsection{Additive martingales}
For any $\gamma\ge 0$ such that $m(\gamma)<\infty$ we can define the {\em additive martingale} $W^{(\gamma)}:=(W_n^{(\gamma)})_{n\in\N_0}$ by
\begin{equation}\label{K:martingale_W_n_gamma}
	W_n^{(\gamma)}:=m(\gamma)^{-n}\sum_{|v|=n}L(v)^{\gamma}
\end{equation}
With the help of Lemma \ref{K:many-to-one-lemma} one can easily check, that $W^{(\gamma)}$ is a nonnegative $F$-martingale with mean one. Hence, by Doob's martingale convergence theorem, we may define $W_\infty^{(\gamma)}:=\lim_{n\rightarrow\infty}W_n^{(\gamma)}$ as its almost sure limit. 
It is well-known, that the limit $W_\infty^{(\gamma)}$ is nondegenerate (which is equivalent to the convergence of $W^{(\gamma)}$ in mean) if and only if 
	\begin{equation}\label{eq:conditions.martingale.convergence}
		\E \bigg[ W_1^{(\gamma)} \log^+  W_1^{(\gamma)}\bigg]<\infty \quad \text{ and } \quad  m(\gamma)\log m(\gamma) > \gamma  m'(\gamma), 
	\end{equation}
	see \cite{Biggins:1977}.

The convergence conditions \eqref{eq:conditions.martingale.convergence} together with $m(\gamma)<\infty$ are satisfied for all $\gamma \le \alpha$ due to our assumptions
\eqref{K:assumption_m} and \eqref{K:assumption_derivative_m}, using that $m$ is convex and continuous, hence decreasing on $[0,\alpha]$.  Note, that in case $\gamma>\alpha$, condition $m(\gamma)<\infty$, as well as the second part of \eqref{eq:conditions.martingale.convergence} in general may not hold, and thus we need to assume these two extra conditions to obtain the convergence in mean of $\mW^{(\gamma)}$.

Throughout our work we will frequently deal with $W:=W^{(\alpha)}$ for $\alpha$ given by assumption \eqref{K:assumption_m}, which for any $n\in\N$ is given by
\begin{equation}\label{K:martingale_W_n}
	W_n=\sum_{|v|=n}L(v)^{\alpha},
\end{equation}
and its a.s. limit $W_\infty$ which is nondegenerate with $\E W_\infty=1$.
It further holds  that
\begin{equation}\label{K:martingale_W_n_limit_branching_property}
	W_\infty=\sum_{\abs{v}=n}L(v)^{\alpha}[W_\infty]_v
\end{equation}
almost surely, for any $n\in\N$; see  Lemma \ref{K:multiplicative_martingale_M_limit_decomposition_lemma} below for the proof of a similar result.

\section{Main results}\label{sect:main results}

In this section we state our main results, the proofs of which will be given in subsequent sections. Equip  the space of bounded continuous complex-valued functions on $\R^3$, $\mathcal{C}_b(\R^3,\C)$ with the supremum norm $\norm{\cdot}_\infty$ and consider $E:=\{ f \in \mathcal{C}_b(\R^3,\C) \, : \, \norm{f}_\infty=1 \}$.

\subsection{Time-dependent solutions}
The general form of the time-dependent solution is given in the following theorem.

\begin{thm}\label{K:time-dependent_solution} Assume \eqref{K:assumption_uniqueness}.
	For any $\phi_0 \in E$, there is a unique function $\phi: [0, \infty) \times \R^3 \to \C$ solving equation \eqref{eq:BoltzmannFourier} with initial value $\phi_0$; with the property that for each $t \in [0,\infty)$, $\phi_t:=\phi(t, \cdot) \in E$ and the mapping  $t \mapsto \phi_t$ is continuous on $[0,\infty)$ and differentiable on $(0,\infty)$ . It is given by
	\begin{equation}\label{K:time-dependent_solution_explicit}
		\phi_t(ro e^{}_3)=\E\bigg[\prod_{v\in\partial \T_t}\phi_0(roL(v)U(v)e^{}_3)\bigg]
	\end{equation}
	for any $(r,o)\in\SpaceRO$ and all $t \ge 0$.
\end{thm}
\begin{rem}\label{K:remark_time-dependent_solution}
	If $\phi_0$ is the characteristic function of a random vector, we are unable to show in general, that the function $\phi_t:\R^3\mapsto \C$, defined by \eqref{K:time-dependent_solution_explicit}, is a characteristic function of a three-dimensional random vector for $t>0$, due to nonlinearity of underlying collisional operator $\Q$. However we can show that $\phi_t$ can be viewed as the restriction of a characteristic function of a suitable random matrix. This implies in particular the continuity of  $\phi_t$, defined by formula \eqref{K:time-dependent_solution_explicit}.

	The relation is as follows. Suppose that $\phi_0$ is the characteristic function of some vector $X=\transpose{(X_1,X_2,X_3)}$ and define a random matrix $\widetilde{X}$ as
	\begin{equation*}
		\widetilde{X}=\begin{pmatrix}
			0 & 0 & X_1\\
			0 & 0 & X_2\\
			0 & 0 & X_3\\
		\end{pmatrix}.
	\end{equation*}
	Denote by $\R^{3\times3}$ a space of $3\times3$ matrices with real entries.  Let $\psi_0:\R^{3\times3}\rightarrow \C$ be the characteristic function of $\widetilde{X}$, which by definition can be written as
	\begin{equation*}
		\psi_0(a)=\E \bigg[e^{i\tr (\transpose{a}\widetilde{X})}\bigg]
	\end{equation*}
	for any $a\in \R^{3\times3}$. Here $\tr (\cdot)$ denotes the trace of a matrix. Observe, that for any $r\ge 0$ and $o\in\OrthogonalGroup$ we have 
	\begin{equation*}
		\phi_0(roe_3)=\E\bigg[e^{i\scalar{roe_3,X}}\bigg]=\E\bigg[ e^{i\tr (r\transpose{o}\widetilde{X})}\bigg]=\psi_0(ro).
	\end{equation*}
	Applying this identity in formula \eqref{K:time-dependent_solution_explicit}, using cyclic property and linearity of the trace, we obtain
	\begin{align*}
		\phi_t(roe_3)&=\E\bigg[\prod_{v\in\partial \T_t}\phi_0(roL(v)U(v)e_3)\bigg]=\E\bigg[\prod_{v\in\partial \T_t}\psi_0(roL(v)U(v))\bigg]
		\\
		&=\E\bigg[ \exp \bigg\{i\sum_{v\in\partial \T_t}\tr (rL(v)\transpose{U(v)}\transpose{o}\widetilde{X}(v))\bigg\}\bigg]
		\\
		&=\E\bigg[ \exp \bigg\{i\sum_{v\in\partial \T_t}\tr (r\transpose{o}L(v)\widetilde{X}(v)\transpose{U(v)})\bigg\}\bigg]
		\\
		&=\E\bigg[ \exp \bigg\{i\tr (r\transpose{o}\sum_{v\in\partial \T_t}\widetilde{X}(v)\transpose{(L(v)U(v))})\bigg\}\bigg],
	\end{align*}
	where $\widetilde{X}(v)$ are independent copies of $\widetilde{X}$. The last term in the above formula can be viewed as the characteristic function of the random matrix $\sum_{v\in\partial \T_t}\widetilde{X}(v)\transpose{(L(v)U(v))}$, evaluated at the point $ro$. 
	
	This relation will be further investigated in the appendix, see Section \ref{sect:geometry}.
\end{rem}

\subsection{Stationary solutions}

In order to identify possible equilibrium states for elastic and inelastic spatial kinetic equations, in this section we discuss the solutions to the stationary equation \eqref{K:stationary_equation}.

The solutions will be described in terms of $W_\infty$ and two auxiliary random functions $V:\OrthogonalGroup\to \R_\ge$ and $Y:\OrthogonalGroup\to \R$ that are functions of the weighted branching process and satisfy the following linear stochastic equations. 
\begin{equation}\label{K:proposition_main_result_invariance_property_V}
	V(o)=\sum_{\abs{v}=n}L(v)^2[V]_v(oU(v)),\quad\as
\end{equation}
and
\begin{equation}\label{K:proposition_main_result_invariance_property_Y}
	Y(o)=\sum_{\abs{v}=n}L(v)[Y]_v(oU(v)),\quad\as,
\end{equation}
respectively. 

\begin{rem}
	While the general set of solutions to equations \eqref{K:proposition_main_result_invariance_property_V} and \eqref{K:proposition_main_result_invariance_property_Y} may depend on the structure of $\U$, we have good knowledge about the particular case of rotationally invariant solutions, \textit{i.e.}, when $V(o)\equiv V$ and $Y(o)\equiv Y$. Observing that $L(v)$ are nonnegative, we may employ the results of  \cite{Alsmeyer+Meiners:2013(9)}, in particular Theorem 4.12, from which it follows that $V=0$ a.s. unless $\alpha=2$, while $Y=0$ a.s. unless $\alpha=1$. Under an extra assumption, which is satisfied if $m(\alpha+\epsilon)<\infty$ for some $\epsilon>0$, it is proved in \cite[Theorem 4.13]{Alsmeyer+Meiners:2013(9)}, that in the case $\alpha=2$, $V=cW_\infty$ a.s. for some $c>0$, and $Y=cW_\infty$ a.s. for some $c \in \R$ in the case $\alpha=1$.  
\end{rem}

One more definition is needed before we can state our result.
	We say that a function $K:\SpaceRO \to \C$ is $(\Support,\gamma)$-invariant, if for all $(r,o) \in \SpaceRO$ it holds that 
	$$ K(r,o)=s^{-\gamma}K(rs,ou)$$
	for all $(s,u)\in\Support$.

\begin{thm}\label{K:proposition_main_result}

	Let $\phi$ be the characteristic function of a random vector in $\R^3$ and  
	assume that \eqref{K:assumption_m} -- \eqref{K:assumption_regular_variation} hold with $\alpha \in (0,2]\setminus\{1\}$. If $\alpha =2$, assume 
	in addition that $\Support=\G \times \U$.
	
	Then $\phi$ is a stationary solution to  \eqref{K:stationary_equation} if and only if  $\phi$ admits the  representation
	\begin{equation}\label{eq:structure of solutions}
		\phi(roe_3) = \E \bigg[ \exp \bigg\{ - \frac12 r^2 V(o) + irY(o) - W_\infty K(r,o) \bigg\}\bigg] \qquad \text{ for all } r\in\R_\ge, o \in \OrthogonalGroup,
	\end{equation}
	where $W_\infty$ is the limit of Biggins' martingale (see \eqref{K:martingale_W_n}),   $V$ and $Y$ are random mappings which solve equations \eqref{K:proposition_main_result_invariance_property_V} and \eqref{K:proposition_main_result_invariance_property_Y} respectively, and $K$ is an $(\Support,\alpha)$-invariant function.
	If $\alpha <1$, then $Y=0$ a.s. and $V=0$ a.s. as long as $\alpha<2$. Further, $K=0$ for $\alpha =2$.
\end{thm}

In \cite[Theorem 2.3]{Bassetti+Ladelli+Matthes:2015}, stationary solutions of \eqref{K:stationary_equation} in the class of characteristic functions of probability measures in the normal domain of attraction of some full $\alpha$-stable distribution on $\R^3$ were studied, assuming in addition the existence of a density for $\U$ (c.f. assumption (H3) in \cite{Bassetti+Ladelli+Matthes:2015}) and that the measures are centered for $\alpha>1$. With our result, we can remove these restrictions: We do not require a density assumption, we do not require the measures to be centered, and, in particular, we can determine all solutions (not necessarily in the domain of attraction of some $\alpha$-stable law) that satisfy the weak tail assumption \eqref{K:assumption_regular_variation}. Of course, this assumption is satisfied by measures in the $\alpha$-stable domain of attraction, but also by any probability measure that has a finite moment of order $\alpha$. 

The structure of the set of solutions is consistent with previous studies on linear fix-point equations, \textit{i.e.} \eqref{eq:homogeneous smoothing}: If we restrict our attention to rotationally invariant solutions, then the equation \eqref{K:stationary_equation} evaluated for the radial part, can be written in the form \eqref{eq:homogeneous smoothing}, with  random vectors $X \in \R^3$ and nonnegative scalar weights $T_j$. The set of solutions is then described by \cite[Theorem 2.1]{Alsmeyer+Meiners:2013(9)}.

\begin{rem}
	Previous work on fixed points of smoothing equations \cite{Meiners+Mentemeier:2017} indicates that the case $\alpha=1$ will require stronger assumptions (finite support for $\rho$ or the existence of a density); this is why we refer in this case to \cite{Bassetti+Ladelli+Matthes:2015}, where for $\alpha=1$ the existence and uniqueness of solutions in the domain of attraction of a Cauchy distribution was shown under a density assumption.
\end{rem} 

\begin{rem}
	The precise additional condition required for the case $\alpha=2$ reads
	\begin{equation}\label{K:assumption_support}\tag{A5}
		\text{There exists }u\in\U\text{ and }s\in\G,\ s \neq 1\text{, such that }(1,u)\in\Support\text{ and }(s,u)\in\Support. 
	\end{equation}
	It will be employed in the final part of the proof of Theorem \ref{K:proposition_main_result} in Section \ref{subsection:proofMainresult}.
\end{rem}

\section{Proofs for the time-dependent case} \label{sect:proofs time dependent}
\begin{proof}[Proof of the Theorem \ref{K:time-dependent_solution}]
	The proof is organized as follows. In Step 1 for $\phi_t$, defined by \eqref{K:time-dependent_solution_explicit}, we derive the representation in terms of a branching random walk $\mZ$, and, using branching property \eqref{K:continuous_BRW_branching_property}, investigate the relation between $\phi_{t+h}$ and $\phi_t$ for any $t,h\ge 0$.  In Step 2 we show, that the mapping $t \mapsto \phi_t(\cdot)$ is continuous. After that we proceed to calculate its derivative w.r.t. to $t$ and explicitly show, that $\phi_t$ is indeed the time-dependent solution to \eqref{eq:BoltzmannFourier}. This will be accomplished in Step 3.

	\textbf{Step 1}: 
	In view of \eqref{K:continuous_BRW}, we may rewrite the right side of \eqref{K:time-dependent_solution_explicit} in terms of $\mZ$ as follows:
	\begin{align*}
		\E\bigg[&\prod_{v\in\partial\T_t}\phi_0(roL(v)U(v)e^{}_3)\bigg]
		=\E\bigg[\exp\log\prod_{v\in\partial\T_t}\phi_0(roL(v)U(v)e^{}_3)\bigg]\\	
		&=\E\bigg[\exp\sum_{v\in\partial\T_t}\log \phi_0(roL(v)U(v)e^{}_3)\bigg]=
		\E\bigg[\exp\int_{\SpaceRO}\log \phi_0(ro y u e^{}_3)\d \mZ_t(y,u)\bigg].
	\end{align*}
	From this, we may in particular infer (using dominated convergence) that $\phi_t \in E$ for all $t \ge 0$.
	
	Now fix any $t,h\ge 0$. Using the branching property \eqref{K:continuous_BRW_branching_property}, we have
	\begin{align}\label{K:proof_proposition_time-dependent_solution_phi_t+h}
		\nonumber\phi_{t+h}(ro e^{}_3)&=\E\bigg(\E\bigg[\exp\int_{}\log \phi_0(ro yu e^{}_3)\d\mZ_{t+h}(y,u)\bigg|\mF_h\bigg]\bigg) \\
		\nonumber&=\E\bigg(\E\bigg[\exp\int_{}\log \phi_0(ro yu e^{}_3)\d\sum_{v\in\partial\T_h}[\mZ_t]_v(y/L(v),(U(v))^{-1}u)\bigg|\mF_h\bigg]\bigg) \\
		\nonumber&=\E\bigg(\E\bigg[\exp\int_{}\sum_{v\in\partial\T_h}\log \phi_0(ro L(v)yU(v)ue^{}_3)\d[\mZ_t]_v(y,u)\bigg|\mF_h\bigg]\bigg) \\
		\nonumber&=\E\bigg(\E\bigg[\prod_{v\in\partial\T_h}\exp\int_{}\log \phi_0(ro L(v)U(v) yu e^{}_3)\d[\mZ_t]_v(y,u)\bigg|\mF_h\bigg]\bigg) \\
		\nonumber&=\E\bigg(\prod_{v\in\partial\T_h}\E\bigg[\exp\int_{}\log \phi_0(ro L(v)U(v) yu e^{}_3)\d[\mZ_t]_v(y,u)\bigg]\bigg) \\
		&=\E\bigg[\prod_{v\in\partial\T_h} \phi_t(ro L(v)U(v)e^{}_3)\bigg].
	\end{align}
	
	\textbf{Step 2}: 
	We now show, that the mapping $t\mapsto \phi_t(\cdot)$ is continuous. 
	Since $\abs{\phi_0(\xi)}\leq 1$ for all $\xi\in\R^3$, for any $t\ge 0$ and $roe_3\in\R^3$ we have
	\begin{equation}\label{K:proof_time-dependent_bounded_by_1}
		\abs{\phi_t(roe_3)}\leq \E\bigg[\prod_{v\in\partial \T_t}\abs{\phi_0(roL(v)U(v)e_3)}\bigg]\leq 1.
	\end{equation}
	For any $t,s$ such that $0\leq t<s$, the number of splits for each particle in $\mZ_t$  during time $[t,s]$ has the Poisson distribution with parameter $s-t$. Since particles in $\mZ_t$ reproduce independently, using \eqref{K:proof_time-dependent_bounded_by_1}, for any $ro e^{}_3\in\R^3$ we have
	\begin{align*}
		\abs{\phi_t(ro e^{}_3)-\phi_s(ro e^{}_3)}&\leq 2\P\{\text{there is at least one split in $(\mZ_t)_{t\ge 0}$ during time [t,s]}\}\\
		&=2\E \big[1-(e^{-(s-t)})^{|\partial\T_t|}\big] \rightarrow 0,
	\end{align*}
	as $s\rightarrow t$, since $|\partial\T_t|$, e.g. the number of alive particles at moment $t$, is finite almost surely for any $t\ge 0$. Since $roe_3$ was arbitrary, we conclude that $\norm{\phi_t - \phi_s}_\infty \to 0$ as $s \to t$.

	\textbf{Step 3}: 
	Having established the continuity of $t\mapsto\phi_t(\cdot)$, we now calculate its derivative. To do that, we first take a closer look on the relation \eqref{K:proof_proposition_time-dependent_solution_phi_t+h}. Consider the amount of splits in branching random walk $\mZ$ on the interval $[0,h]$. Since at moment zero there is only one living particle and its lifetime is exponentially distributed, we have
	\begin{align*}
		&\P\{\text{there are no splits during $[0,h]$}\}=e^{-h}
		\\
		&\P\{\text{there is one split during $[0,h]$}\}=h
		\\
		&\P\{\text{there are more than one splits during $[0,h]$}\}=1-e^{-h}-h.
	\end{align*}
	Conditioning on each of these three events, we may rewrite the formula \eqref{K:proof_proposition_time-dependent_solution_phi_t+h} as
	\begin{align*}
		\phi_{t+h}(ro e^{}_3)
			&=\phi_t(ro e^{}_3)e^{-h}+\Q(\phi_t,\phi_t)(ro e^{}_3)h + \E\bigg[\prod_{v\in\partial\T_h} \phi_t(ro L(v)U(v)e^{}_3)(1-e^{-h}-h)\bigg].
	\end{align*}
	We now have all the ingridients to calculate the derivative of $t \mapsto \phi_t(roe_3)$, at any fixed $roe_3 \in \R^3$:
	\begin{align*}
		&\frac{\partial}{\partial t}\phi_t(ro e^{}_3)
		\\
		&=\lim_{h\rightarrow0}\frac{\phi_{t+h}(ro e^{}_3)-\phi_t(ro e^{}_3)}{h} 
		\\
		&=\lim_{h\rightarrow0}\frac{\phi_t(ro e^{}_3)(e^{-h}-1)+\Q(\phi_t,\phi_t)(ro e^{}_3)h + \E\bigg[\prod\limits_{v\in\partial\T_h} \phi_t(ro L(v)U(v)e^{}_3)\bigg](1-e^{-h}-h)}{h} 
		\\
		&=\lim_{h\rightarrow0} \frac{(e^{-h}-1)}{h}\phi_t(ro e^{}_3) +\Q(\phi_t,\phi_t)(ro e^{}_3)+\lim_{h\rightarrow0}\frac{1-e^{-h}-h}{h}\E\bigg[\prod\limits_{v\in\partial\T_h} \phi_t(ro L(v)U(v)e^{}_3)\bigg]
		\\
		&= -\phi_t(ro e^{}_3) + \Q(\phi_t,\phi_t)(ro e^{}_3).
	\end{align*}
	Thus, $\phi_t$ solves the equation \eqref{eq:BoltzmannFourier}. Its uniqueness follows from the Picard-Lindel\"of theorem, see e.g. \cite[Proposition 2.2]{Bassetti+Ladelli+Matthes:2015}.
\end{proof}

\section{Proofs for the stationary case}\label{sect:proofs stationary}

The proof of Theorem \ref{K:proposition_main_result} is much more involved and requires a couple of intermediary steps. The proof of the theorem will be given at the end of this section and relies on Lemmas \ref{K:proposition_solution_via_multiplicative_martingale_M}, \ref{K:proposition_levy_triplet_nu_characteristics}, \ref{K:proposition_uniform_boundedness_of_gamma_o}, \ref{K:proposition_uniform_boundedness_of_mu_o} and \ref{K:proposition_uniform_boundedness_of_theta_o}.  As in previous works on branching fixed point equations (see \cite{Meiners+Mentemeier:2017} and references therein), we begin with introducing a multiplicative martingale, which will be studied in the first part of this section. 

\subsection{Multiplicative martingales and random L\'evy-Khintchine exponents}\label{K:sec:multiplicative_martingales_and_random_levy-khinchine_exponents}

{It will be stipulated throughout this section that $\phi$ is a characteristic function of an $\R^3$-valued random variable satisfying \eqref{K:stationary_equation}; and that assumptions \eqref{K:assumption_m} -- \eqref{K:assumption_regular_variation} are in force.}

 For $r\in\Rplus$ and $o\in\OrthogonalGroup$ we define a stochastic process $M=(M_n(r,o))_{n\in\N_0}$ as
\begin{equation}\label{K:multiplicative_martingale_M}
	M_n(r,o):=\prod_{|v|=n}\phi(ro L(v)U(v)e^{}_3).
\end{equation}
We suppress the dependence of $M$ on $\phi$ in this notation. 
Since $\phi$ satisfies \eqref{K:stationary_equation}, for any $n\in\N$ we have
\begin{align*}
	\E \bigg[M_{n+1}(r,o)|F_n\bigg]&= \E\bigg[\prod_{|v|=n+1}\phi(ro L(v)U(v)e^{}_3)\big|F_n\bigg]\\
	&=\E\bigg[\prod_{|v|=n}\prod_{j=1,2}\phi(ro L(v)U(v)R_j(v)O_j(v)e^{}_3)\big|F_n\bigg]\\
	&=\prod_{|v|=n}\E\bigg[\phi(ro L(v)U(v)e^{}_3)\big|F_n\bigg]=M_n(r,o).
\end{align*}
Hence, for each fixed $r \in \Rplus$, $o \in \OrthogonalGroup$
\begin{equation*}
	M_n(r,o)=\prod_{|v|=n}\phi(ro L(v)U(v)e^{}_3)
\end{equation*}
is a bounded $L^1$-martingale. Thus its a.s. limit, which we denote by $M_\infty(r,o)$, satisfies
\begin{equation}\label{eq:Disintegration}
	\phi(ro e^{}_3)=\E[ M_\infty(r,o)].
\end{equation}
The above identity \eqref{eq:Disintegration} is the main reason to study the martingale $M$, for it will allow us to determine the structure of $\phi$.
The basis of our analysis will be the following Proposition \ref{prop:SimultaneousConvergence} stating that almost surely, for each fixed $o \in \OrthogonalGroup$, $r\mapsto M_\infty(r,o)$ is the one-dimensional characteristic function of an infinitely divisible random variable. This technique goes back to \cite[Theorem 1]{Caliebe2003}. Note, that for this claim we will extend the definition of $M$ in \eqref{K:multiplicative_martingale_M} for all $r\in\R$. Observe that, in contrast e.g. to \cite[Lemma 3.1]{Meiners+Mentemeier:2017}, the particular structure of $\Q$ does not allow us to obtain that $M_\infty$ is a three-dimensional characteristic function. Instead we obtain, that along each fixed direction $o e_3$ it is a one-dimensional characteristic function of an infinitely divisible law. Note, that there is in general no Cram\'er-Wold device for infinite divisibility, see e.g. the counterexample given in \cite{Dwass1957}.

For the following considerations denote
\begin{equation}\label{K:formula:bold_L}
	\mathbf{L}:=(L(v),U(v) : j\in\{1,2\})_{v\in\T}.
\end{equation}
The Doob martingale convergence theorem yields that for each fixed pair $(r,o)$, there is a set $N_{r,o} \subset (\SpaceRO)^{\T}$ with $\P(\mathbf{L} \in N_{r,o}^c=1)$ such that for all $\mathbf{l}=(l(v),u(v) : j\in\{1,2\})_{v\in\T} \in N_{r,o}^c$  we have the convergence
\begin{equation}\label{K:formula:M_n(r,o,l)_converges_to_M_infinity(r,o,l)}
	M_n(r,o,\mathbf{l}) =\prod_{|v|=n}\phi(ro l(v)u(v)e^{}_3) \to M_\infty(r,o,\mathbf{l}).
\end{equation}

Our aim is to obtain a \emph{global} null set $N$, such that we have simultaneous convergence for all $(r,o)$ whenever $\mathbf{l} \in N^c$. Namely, we are going to prove the following result.

\begin{prop}\label{prop:SimultaneousConvergence} 
	Assume \eqref{K:assumption_m} -- \eqref{K:assumption_regular_variation}.  There is a set $N$ with $\P(\mathbf{L} \ \in N)=0$, such that for all $\mathbf{l} \in N^c$ it holds
	$$ M_n(r,o,\mathbf{l}) \to M_\infty(r,o,\mathbf{l})$$
	simultaneously for all $r \in \R$ and $o \in \OrthogonalGroup$. The mapping $(r,o) \mapsto M_\infty(r,o,\mathbf{l})$ is continuous.
	Moreover, for each fixed $o \in \OrthogonalGroup$, the mapping $r \mapsto M_\infty(r,o,\mathbf{l})$ is the characteristic function of an infinitely divisible random variable.
\end{prop}

The proof of the proposition will consist of several steps. 
As the first step, we want to show that for each fixed $o \in \OrthogonalGroup$, the convergence holds for all $r \in \R$, and the limit is a characteristic function. 

We follow the approach by \cite{Caliebe2003}. Fix $o \in \OrthogonalGroup$. If $X$ is a random vector with characteristic function $\phi$, then we may consider 
$$r \mapsto \begin{cases}
	\phi(roe_3) & r \ge 0,\\
	\phi\bigg(|r|(-o)e_3\bigg) & r <0
\end{cases}$$ as the characteristic function of $\scalar{e_3,\transpose{o} X}$. The stationary equation
\begin{equation}\label{distributional equation 1D}
	\phi(roe_3) = \E \bigg[\prod_{|v|=1} \phi(rL(v)oU(v))e_3\bigg]
\end{equation}
then corresponds to the following distributional equation
$$ \scalar{e_3, \transpose{o} X} \stackrel{d}{=} \sum_{|v|=1} L(v)\scalar{U(v)e_3,o^\top X(v)},$$
where $\stackrel{d}{=}$ means that both sides have the same law, and $X(v)$ are i.i.d.\ copies of $X$, which are also independent of $\{ L(v), U(v) \, : \, |v|=1\}$.

\begin{lem}
	For each $u \in \R$, we have that $(G_n(u))_{n \in \N}$, defined by
	$$ G_n(u) := \P \bigg( \sum_{|v|=n} L(v) \scalar{U(v)e_3,o^\top X(v)} \le u \, \bigg| \, F_n \bigg), $$
	is an $F$-martingale.	
\end{lem}

\begin{proof}
	By definition, $G_n(u)$ is adapted to the filtration $F$, and integrable.
	Considering the martingale property, we use the tower property of conditional expectation and then \eqref{distributional equation 1D}:
	\begin{align*}
		\E \bigg( G_{n+1}(u) \, \big| \, F_n \bigg) =& \P \bigg( \sum_{|w|=n} \sum_{|v|=1} L(w) [L(v)]_w \scalar{U(w) [U(v)]_w e_3,o^\top [X(v)]_w} \le u \, \bigg| \, F_n \bigg) \\
		=& \P \bigg( \sum_{|w|=n}  L(w)  \scalar{U(w)  e_3,o^\top X(w)} \le u \, \bigg| \, F_n \bigg) = G_n(u) \quad \as
	\end{align*}
\end{proof}

\begin{lem}\label{lem:Minfintelydivisible}
	For each $o \in \OrthogonalGroup$, there is a null set $N_o$ such that for all $\mathbf{l} \in N_o^c$, the sequence of mappings $ r \mapsto M_n(r,o,\mathbf{l})$ converges to the characteristic function of an infinitely divisible random variable $Y^o(\mathbf{l})$. Hence, on $N_o^c$, there is a version of $M_\infty(r,o, \mathbf{l})$ such that $ r \mapsto M_\infty(r,o,\mathbf{l})$ is the characteristic function of an infinitely divisible random variable.
\end{lem}

\begin{proof}
	Since for each $u \in \R$, $G_n(u)$ is a martingale by the previous lemma that is moreover bounded (by 1), we obtain the existence of an $\as$-limit $G_\infty(u)$. Obviously, we also have the simulatenous convergence $G_n(u,\mathbf{l}) \to G_\infty(u, \mathbf{l})$ for all $u \in \mathbb{Q}$ and $\mathbf{l} \in N_o^c$, where $N_o$ is set with the property $\P(\mathbf{L} \in N_o^c)=1$. 
	
	We now introduce the function $\widetilde{G}(u,\mathbf{l}):=\inf_{s>u,s\in\mathbb{Q}} G_\infty(s,\mathbf{l})$. 
	Since each $G_n(\cdot,\mathbf{l})$ is nondecreasing with $\lim_{u \to -\infty} G_n(u,\mathbf{l})=0$ and $\lim_{u \to \infty} G_n(u,\mathbf{l})=1$, the same holds true for $\widetilde{G}(\mathbf{l})$ as well. Further, $\widetilde{G}(\mathbf{l})$ is right-continuous with left limits. At every $u \in \mathbb{Q}$ that is a continuity point of $\widetilde{G}(\mathbf{l})$, it coincides with $G_\infty(\mathbf{l})$. We conclude that $\widetilde{G}(\mathbf{l})$ is the distribution function of a random variable $Y(\mathbf{l})$, and $G_n(u,\mathbf{l}) \to \widetilde{G}(u,\mathbf{l})$ at every continuity point of $\widetilde{G}(\mathbf{l})$.
	
	Moreover, as in \cite[Section 4.1]{Caliebe2003} we obtain that for each fixed $\mathbf{l}$, $u \mapsto G_n(u,\mathbf{l})$ is the distribution function of $Y_n^o:=\sum_{|v|=n} l(v) \scalar{o(v)e_3,o^\top X(v)}$, and $(Y_n^o)_{n \ge 1}$ can be considered as a sequence of row sums in a triangular  array. By \eqref{K:supremum_L(v)_goes_to_0}, it holds that $\lim_{n \to \infty} \sup_{|v|=n} \mathbf{l}(v) = 0$ for all $\mathbf{l} \in N_o^c$ (possibly after insecting with another null set). Hence, for all $\mathbf{l} \in N_o^c$ $(Y_n^o(\mathbf{l}))_{n \ge 0}$ constitutes a sequence of row sums in a triangular null array, and we have just shown (in the above paragraph) that it converges in distribution to a random variable $Y^o(\mathbf{l})$. It follows that $Y^o(\mathbf{l})$ is infitely divisble.
	
	For each fixed $\mathbf{l} \in N_o^c$, $r \mapsto M_n(r,o,\mathbf{l})$ is the characteristic function of $Y_n^o(\mathbf{l})$, hence it converges by the above considerations to the characteristic function of $Y^o(\mathbf{l})$. 
\end{proof}

As the next step, we consider simultaneous convergence over $o \in \OrthogonalGroup$.
Fix a countable dense subset $D \subset \OrthogonalGroup$ and consider the null set $N^c:=\bigcap_{o \in D} N_o^c$. 
We will rely on an argument using the Arzel\`a-Ascoli theorem.

\begin{lem}\label{lem:ArzelaAscoli}
	For each fixed $r\in\R$ and $\mathbf{l} \in N^c$, the set of continuous functions on $\OrthogonalGroup$
	$$ F:=\bigg\{ o \mapsto M_n(r,o,\mathbf{l}) \, : \, n \in \N \bigg\}$$ 
	is equicontinuous and pointwise bounded.
\end{lem}

\begin{proof}
	Throughout the proof we assume that $r\in\R$ is fixed.
	Since $M_n$ is a finite product of characteristic functions, we immediately obtain that each $M_n$ is continuous and uniformly bounded by 1.
	
	To prove the equicontinuity, we use the representation 
	$$ M_n(r,o,\mathbf{l})= \E\bigg[ \exp \bigg\{ i \sum_{|v|=n} r l(v) \scalar{u(v)e_3,o^\top X(v)} \bigg\}\bigg], $$
	valid for all $r \in \R$ and $o \in \OrthogonalGroup$, where the expectation is taken only with respect to the family $X(v)$ of i.i.d.\ copies of $X$. We start our estimations by using the inequalities $|e^{ix}-e^{iy}| \le |1-e^{i(y-x)}|$ and $|1-e^{ix}| \le 2 (|x| \wedge 1)$.
	\begin{align*}
		~\abs{M_n(r,u,\mathbf{l})-M_n(r,o,\mathbf{l})}~
		\le&~ \E\bigg[ \bigg| 1- \exp \bigg( i \sum_{|v|=n} r l(v) \scalar{u(v)e_3,(o^\top -u^\top)X(v)} \bigg) \bigg|\bigg]  \\
		\le&~ 2 \E  \bigg[ \bigg| \sum_{|v|=n} r l(v) \scalar{u(v)e_3,(o^\top -u^\top)X(v)} \bigg| \wedge 1 \bigg] 
	\end{align*}
	Using next the triangle inequality, the Cauchy-Schwarz-inequality and the property of the matrix norm, we can further bound
	\begin{align*} 
		\abs{M_n(r,u,\mathbf{l})-M_n(r,o,\mathbf{l})} ~ \le&~ 2 \E  \bigg[ \bigg| \sum_{|v|=n} r l(v) \big|\scalar{u(v)e_3,(o^\top -u^\top)X(v)}\big| \bigg| \wedge 1 \bigg]   \\
		\le&~ 2 \E  \bigg[ \bigg( \sum_{|v|=n} \abs{r} l(v)  \norm{X(v)}  \norm{o^\top -u^\top} \bigg) \wedge 1 \bigg]    
	\end{align*}
	
	Now we separate according to whether $ \sum_{|v|=n} \abs{r}  l(v)  \norm{X(v)}>K$ or not, where $K$ is to be chosen below. In the first case, we bound the expectation by the probability of the corresponding event, while for the second case we bound the sum by $K$. We obtain
	
	\begin{align*} 
		\abs{M_n(r,u,\mathbf{l})-M_n(r,o,\mathbf{l})} ~\le&~ 2 \P\bigg(\sum_{|v|=n} \abs{r} l(v) \norm{X(v)} >K\bigg) + 2 K \norm{o-u},	
	\end{align*}
	where we have used as well that $\norm{u^\top-o^\top}=\norm{(u-o)^\top}=\norm{u-o}$.

	Now fix $\epsilon>0$. 	By Lemma \ref{lem:regvar} below, for given $\epsilon>0$, there is $K=K(\abs{r})$, such that $$\sup_n \P\bigg(\sum_{|v|=n} \abs{r} l(v)\norm{X(v)} >K\bigg) < \epsilon/4.$$ Set $\delta:=\frac{\epsilon}{4K}$. Then for all $u,o \in \OrthogonalGroup$ satisfying $\norm{u-o} < \delta$ it holds 
	
	\begin{align*}
		~\abs{M_n(r,u,\mathbf{l})-M_n(r,o,\mathbf{l})}~\le&~ 2 \epsilon/4 + 2 K \norm{u-o}< \epsilon.  
	\end{align*}
	This proves the equicontinuity.
\end{proof}

\begin{lem}\label{lem:regvar}
	For each $\mathbf{l} \in N^c$, for any given $\epsilon>0$ and any fixed $r>0$, there is $K$ such that  $$\sup_n \P\bigg(\sum_{|v|=n} r l(v) |X(v)| >K\bigg) < \epsilon/4.$$ \end{lem}

\begin{proof}
	We rely on assumption \eqref{K:assumption_regular_variation}. 	It suffices to obtain a bound for random vectors $X$ with Pareto tails, \textit{i.e.}, $ P(\norm{X}>t)\sim Ct^\alpha>0$; since by assumption \eqref{K:assumption_regular_variation}, any other stationary solution of interest in the paper would be stochastically dominated by such a random variable. Here, $f(t)\sim g(t)$ denotes asymptotic equivalence, \textit{i.e.}, $\lim_{t \to \infty} \frac{f(t)}{g(t)}=1$.
	
	Let $r>0$ be fixed. By \cite[Lemma 3.3]{Jessen2006}, for all sufficiently large $K>0$
	\begin{align*}
		\P\bigg(\sum_{|v|=n} r l(v) \norm{X(v)} >K\bigg)~ 
		\sim~ \P(\norm{X} >K) \bigg(\sum_{|v|=n} r^\alpha l(v)^\alpha\bigg)
	\end{align*}
	For any $\mathbf{l} \in N^c$, we have the convergence $\lim_{n \to \infty} \sum_{|v|=n} r^\alpha l(v)^\alpha = r^{\alpha}W_\infty(\mathbf{l})$, where $W_\infty(\mathbf{l})$ denotes a realization of $W_\infty$. In particular, $\sup_{n \in \N} \sum_{|v|=n} r^\alpha l(v)^\alpha < \infty$; hence there is a constant $C'=C'(\mathbf{l})$ such that
	$$ \sup_{n \in \N} \P\bigg(\sum_{|v|=n} r^\alpha l(v)^\alpha \norm{X(v)}^\alpha >K\bigg) \le C' \P(\norm{X}^\alpha >K) .$$
	The assertion follows by choosing $K$ large enough.
\end{proof}

\begin{cor}\label{cor:ArzelaAscoli}
	For each compact set $A \subset \R$ and for  $\mathbf{l} \in N^c$, the set of continuous functions on $A \times \OrthogonalGroup$
	$$ F:=\bigg\{ r,o \mapsto M_n(r,o,\mathbf{l}) \, : \, n \in \N \bigg\}$$ 
	is equicontinuous and pointwise bounded.
\end{cor}

\begin{proof}
	Since set $A$ is compact, Lemma \ref{lem:ArzelaAscoli} immediately implies, that the family
	\begin{equation*}
		\bigg\{ o \mapsto M_n(r,o,\mathbf{l}) \, : \, n \in \N, r \in A \bigg\}
	\end{equation*}
	is equicontinuous (in both variables $r$ and $n$) and pointwise bounded. Since $\OrthogonalGroup$ is also compact, one can check, that the family
	\begin{equation*}
		\bigg\{ r \mapsto M_n(r,o,\mathbf{l}) \, : \, n \in \N, o \in \OrthogonalGroup \bigg\}
	\end{equation*}
	is equicontinuous and pointwise bounded by using the same techniques, as in Lemma  \ref{lem:ArzelaAscoli}. We refrain from providing details. 
	
	Now fix an arbitrary $\epsilon>0$. By above arguments, there exist $\delta_1$ and $\delta_2$, which depend only on $\epsilon$, such that for all $r,s\in A$ and $o,u\in\OrthogonalGroup$ satisfying $\abs{r-s}<\delta_1$ and $\norm{o-u}<\delta_2$, the relations
	\begin{equation}\label{K:proof_equicont_family_relations_M_n}
		\abs{M_n(r,o,\mathbf{l}) - M_n(s,o,\mathbf{l})}<\frac{\epsilon}{2}\quad\text{and}\quad\abs{M_n(s,o,\mathbf{l}) - M_n(s,u,\mathbf{l})}<\frac{\epsilon}{2}
	\end{equation}
	both hold. Let $d$ be a metric on $\R\times\OrthogonalGroup$, defined as
	\begin{equation*}
		d((r,o),(s,u)) := \max\{\abs{r-s},\norm{o-u}\},
	\end{equation*}
	and notice, that for all $r,s\in\R$ and $o,u\in\OrthogonalGroup$ the relations $\abs{r-s}<\delta_1$ and $\norm{o-u}<\delta_2$ imply the existence of $\delta:=\max\{\delta_1,\delta_2\}$, such that $d((r,o),(s,u))<\delta$. In view of \eqref{K:proof_equicont_family_relations_M_n}, by the merit of triangle inequality we obtain
	\begin{equation*}
		\abs{M_n(r,o,\mathbf{l}) - M_n(s,u,\mathbf{l})} < \epsilon,
	\end{equation*}
	which proves equicontinuity. Boundedness follows from the properties of $M$, discussed at the beginning of Lemma \ref{lem:ArzelaAscoli}.
\end{proof}

\begin{proof}[Proof of Proposition \ref{prop:SimultaneousConvergence}]
	Using Corollary \ref{cor:ArzelaAscoli}, we may apply the Arzel\`a-Ascoli theorem to obtain for any fixed compact set $K \subset \R$ the existence of a continuous function $(r,o)\mapsto g(r,o,\mathbf{l})$ that is the limit of a subsequence of functions in $F$. But we know, that the full sequence converges on the countable dense set $K \times D$, hence the full sequence converges to the continuous function $g(\cdot, \mathbf{l})$ w.r.t. to uniform convergence for continuous functions on $K \times \OrthogonalGroup$. 
	Taking a countable union of compact sets that covers $\R$ we conclude that there is a null set $N$, such that for all $\mathbf{l} \in N^c$
	$$ M_n(r,o,\mathbf{l}) \to M_\infty(r,o,\mathbf{l})$$
	simultaneously for all $r \in \R$ and $o \in \OrthogonalGroup$, and the mapping $(r,o) \mapsto M_\infty(r,o,\mathbf{l})$ is continuous. 
	
	The remaining assertion of the proposition has already been proved in Lemma \ref{lem:Minfintelydivisible}.
\end{proof}

\begin{rem}\label{K:rem:null_set}
	Throughout the paper a statement, that a given property holds almost surely, should be understood in the way, that this property holds on the set $H^c$, unless stated otherwise, where the set $H^c$ is given as follows. Firstly, since most our results rely on $M_\infty$ being a one-dimensional characteristic function, one naturally requires $N\subset H$. To comfortably cover all instances of the formulation in question, let $H^c\subset \mathbf{L}$ be defined for the rest of the paper as
	\begin{equation}\label{K:almost_sure_set}
		H^c=N^c\cap O^c\cap P^c
	\end{equation}
	where $O^c$ denotes the set, on which \eqref{K:supremum_L(v)_goes_to_0} holds, and $P^c$ is the set, on which the limit of an additive martingale $W$ defined in \eqref{K:martingale_W_n} exists. All three components are of probability 1, hence $\P(\mathbf{L}\in H^c)=1$.
\end{rem}

\begin{lem}\label{K:multiplicative_martingale_M_limit_decomposition_lemma} 
For each $n \in \N$, it holds 
	\begin{equation}\label{K:multiplicative_martingale_M_limit_decomposition}
		M_\infty(r,o) = \prod_{\abs{v}=n}[M_\infty]_v (rL(v),oU(v)),
	\end{equation}
	almost surely, simultaneously for all $r \in \R$ and $o \in \OrthogonalGroup$.
\end{lem}

\begin{proof}
Decomposing at generation $n$, we obtain  that
	\begin{align*}
		M_\infty(r,o)&=\lim_{k\rightarrow\infty}\prod_{|v|=n+k}\phi(roL(v)U(v)e^{}_3)\\
		&=\lim_{k\rightarrow\infty}\prod_{|v|=n}\prod_{|w|=k}\phi(roL(v)U(v)[L(w)]_v[U(w)]_ve^{}_3)\\
		&=\prod_{\abs{v}=n}[M_\infty]_v(rL(v),oU(v)),\quad \as
	\end{align*}
	Note that the size of generation $n$ is $2^n$, in particular finite, hence we may exchange limit and product. Further, we may obtain the simultaneous convergence for all arguments from a dense subset of $\R\times\OrthogonalGroup$. By Proposition \ref{prop:SimultaneousConvergence} there is a set of measure 1 on which $M_\infty$ as well as $([M_\infty]_v)_{|v|=n}$ are continuous functions, hence the identity holds for all $r $ and $o $ simultaneously.
\end{proof}

We say that $\Psi$ is a random L\'evy-Khintchine exponent with $F_\infty$-measurable random L\'evy-triplet $(\mu, \Theta, \nu)$ if 
	\begin{equation}\label{K:characteristic_exponent_Psi}
	\Psi(r)=ir\mu - \frac12 r^2\Theta + \int_{\R} \bigg(e^{iry}-1-iry\1_{[0,1]}(\abs{y})\bigg)\nu(\d y).
\end{equation}
where $\mu$ and $\Theta$ are $F_\infty$-measurable random variables taking values in $\R$ and $[0,\infty)$, respectively, and $\nu$ is an $F_\infty$-measurable random L\'evy measure.
Combining Proposition \ref{prop:SimultaneousConvergence} and Lemma \ref{K:multiplicative_martingale_M_limit_decomposition_lemma}, we may now obtain equations for the components of the random L\'evy-Khintchine exponent pertaining to $M_\infty(\cdot,o)$.

\begin{prop}\label{K:proposition_solution_via_multiplicative_martingale_M} 
	For every fixed $o\in\OrthogonalGroup$, the mapping $r\mapsto M_\infty(r,o)$ admits the representation
	\begin{equation}\label{K:multiplicative_martingale_M_exp_Psi}
		M_\infty(\cdot,o)=e^{\Psi^{o}(\cdot)}, 
	\end{equation}
	where $\Psi^{o}$ is a random L\'evy-Khintchine exponent with $F_\infty$-measurable random L\'evy-triplet $(\mu^o, \Theta^o, \nu^o)$. The random variable $\Theta^o$ satisfies
%
	\begin{equation}\label{K:levy_triplet_theta_branching_property}
		\Theta^{o}=\sum_{\abs{v}=n}L(v)^2[\Theta]_v^{oU(v)};\quad \as\\
	\end{equation}
	while $\nu^o$ satisfies  for any $B\in\Rnot$
	\begin{equation}\label{K:levy_triplet_nu_branching_property}
		\nu^{o}(B)=\sum_{\abs{v}=n}[\nu]_v^{oU(v)}(L(v)^{-1}B),\quad\as
	\end{equation}
\end{prop}

\begin{rem}\label{rem:represenation:phi:psi}
		By Proposition \ref{K:proposition_solution_via_multiplicative_martingale_M}, using \eqref{eq:Disintegration}, we have that $\phi$ admits the representation
\begin{equation}\label{K:solution_E_of_exp_Psi}
	\phi(ro e^{}_3)=\E[ e^{\Psi^{o}(r)}]
\end{equation}
for any $r\in\Rplus$ and $o\in\OrthogonalGroup$.
\end{rem}

\begin{proof}[Proof of Proposition \ref{K:proposition_solution_via_multiplicative_martingale_M}]
	By Proposition \ref{prop:SimultaneousConvergence}, for all $\mathbf{l} \in N^c$, $r \mapsto M_\infty(r,o,\mathbf{l})$ is the characteristic function of an infinitely divisible law. Therefore, it admits the representation $M_\infty(r,o, \mathbf{l})=\exp(\Psi^{o}(r,\mathbf{l}))$, for a (deterministic, $\mathbf{l}$-dependent) L\'evy-Khintchine exponent
	\begin{equation}\label{K:proof:Psior}
		\Psi^{o}(r)=ir\mu^{o} - \frac12 r^2\Theta^{o} + \int_{} \bigg(e^{ir y}-1-ir y\1_{[0,1]}(\abs{y})\bigg)\nu^{o}(\d y).
	\end{equation} 
	Note, that $\Psi^o(r)= \log M_\infty(r,o)$ is a measurable function 
	By \cite[Theorem 13.28]{Kallenberg:1997}, we have the vague convergence
	$$ \sum_{|v|=n} \P\bigg(l(v)\scalar{u(v)e_3,o^\top X(v)} \in \cdot \bigg) \rightarrow \nu^o(\mathbf{l})(\cdot)$$
	from which we may deduce the $F_\infty$-measurability of $\nu^o$.
%
	The corresponding formulae for $\mu^o$ and $\Theta^o$ in \cite[Theorem 13.28]{Kallenberg:1997} involve the choice of a continuity point of $\nu^o$; which may be random at this stage. We therefore postpone the proof of measurability of $\mu^o$ and $\Theta^o$ to Lemmata \ref{K:proposition_uniform_boundedness_of_mu_o} and \ref{K:proposition_uniform_boundedness_of_theta_o}, respectively. Note that in both cases the proofs will rely only on the measurability and properties of $\nu^o$, and the results of Proposition \ref{prop:SimultaneousConvergence}.

	By combining  \eqref{K:multiplicative_martingale_M_limit_decomposition} and \eqref{K:multiplicative_martingale_M_exp_Psi}, we obtain
	\begin{equation}\label{K:characterisric_exponent_Psi_branching_property}
		\Psi^{o}(r)=\sum_{\abs{v}=n}[\Psi]_v^{oU(v)}(rL(v)),\quad \as
	\end{equation}
	Expanding the right-hand side of \eqref{K:characterisric_exponent_Psi_branching_property} yields
	\begin{align*}
		\Psi^{o}(r)&=ir\mu^{o} - \frac12 r^2\Theta^{o} + \int_{} \bigg(e^{ir y}-1-ir y\1_{[0,1]}(y)\bigg)\nu^{o}(\d y)\\
		&=ir\sum_{\abs{v}=n}L(v)[\mu]_v^{oU(v)} - \frac12 r^2\sum_{\abs{v}=n}L(v)^2[\Theta]_v^{oU(v)}\\ 
		&+ \sum_{\abs{v}=n}\int_{} \bigg(e^{irL(v) y}-1-irL(v) y\1_{[0,1]}(y)\bigg)[\nu]_v^{oU(v)}(\d y)\\
		&=ir\sum_{\abs{v}=n}L(v)[\mu]_v^{oU(v)} - \frac12 r^2\sum_{\abs{v}=n}L(v)^2[\Theta]_v^{oU(v)}\\
		&+ \sum_{\abs{v}=n}\int_{} \bigg(e^{irL(v) y}-1-irL(v) y\1_{[0,1]}(L(v) y)\bigg)[\nu]_v^{oU(v)}(\d y)\\
		&- \sum_{\abs{v}=n}\int_{} \bigg(irL(v) y\1_{[0,1]}(y)-irL(v) y\1_{[0,1]}(L(v) y)\bigg)[\nu]_v^{oU(v)}(\d y),
	\end{align*}
	where all equations hold almost surely. The last term in the above formula contributes to the random shift, thus,  using uniqueness of the L\'evy-triplet, we infer
	\begin{equation}\label{K:brancing_property_theta}
		\Theta^{o}=\sum_{\abs{v}=n}L(v)^2[\Theta]_v^{oU(v)},\quad \as\\
	\end{equation}
	and
	\begin{equation}\label{K:levy_triplet_integral_nu_branching_property}
		\int h(y)\nu^{o}(\d y)=\sum_{\abs{v}=n}\int_{} h(L(v)y)[\nu]_v^{oU(v)}(\d y),\quad \as
	\end{equation}
	for all nonnegative measurable functions $h$ on $\Rnot $. By taking $h(x)=\1_{B}(x)$ for any $B\subset\Rnot$, we further obtain
	\begin{equation*}
		\nu^{o}(B)=\sum_{\abs{v}=n}[\nu]_v^{oU(v)}(L(v)^{-1}B),\quad\as
	\end{equation*}
\end{proof}

\subsection{Determining the random L\'evy measure}

In the next major step we consider properties of the random L\'evy-measure $\nu^o$.
An important role here is played by  assumption \eqref{K:assumption_regular_variation}, under which the map $r\mapsto r^\alpha\nu^o((r,\infty))$ is bounded for any $o\in\OrthogonalGroup$, as will be proved in the next lemma.
\begin{lem}\label{K:levy_triplet_nu_boundedness_lemma}
	There is a constant $C<\infty$ such that for all $o \in \OrthogonalGroup$ and $r>0$,
		\begin{equation*}
		\nu^o((\tfrac{1}{r},\infty))\leq C r^\alpha W_\infty \quad \text{a.s., \quad and \quad} 	\nu^o((-\infty,-\tfrac{1}{r}))\leq C r^\alpha W_\infty \quad \text{a.s.}.
	\end{equation*}
\end{lem}
\begin{proof}
	We fix any $o\in\OrthogonalGroup$ and recall that $M_\infty(r,o)$ is the limit of characteristic functions of sums in a triangular array $\sum_{\abs{v}=n}L(v)[X]_v^{oU(v)}$, in which all summands in any generation $\{\abs{v}=n\}$ are independent and infinitesimal, and $\nu^o$ is the L\'evy-measure in the characteristic exponent $\Psi^o$ of $M_\infty(r,o)$.
	By Theorem 13.28, part (i) in \cite{Kallenberg:1997}
	 it follows that
	\begin{equation}\label{K:vague_convergence}
		\sum_{\abs{v}=n}\E\bigg[ f\bigg(L(v)[X]_v^{oU(v)}\bigg)\bigg]\rightarrow \int f(y) \nu^o(\d y)\quad \as
	\end{equation}
	as $n\rightarrow\infty$, for any continuous and compactly supported function $f$ on $\overline{\R}\setminus\{0\}$. By choosing a smooth function $f \le 1$ that dominates $\mathds{1}_{(\frac{1}{r},\infty]}$ and vanishes outside $(\frac{1}{2r},\infty]$, we obtain
	\begin{equation}\label{K:proof_rewriting_vague_convergence}
		\liminf_{n \to \infty} \sum_{\abs{v}=n}\P\bigg([X]_v^{oU(v)}>\tfrac{1}{2rL(v)}\bigg) \ge  \liminf_{n \to \infty} \sum_{\abs{v}=n}\E\bigg[ f\bigg(L(v)[X]_v^{oU(v)}\bigg)\bigg]  \ge \nu^o((\tfrac{1}{r},\infty]),
	\end{equation}
	as $n\rightarrow\infty$, where all inequalities hold almost surely. Let us now take a look on the left side in \eqref{K:proof_rewriting_vague_convergence}. Direct calculations yield
	\begin{align}\label{K:sum_of_tails}
		\nonumber&\sum_{\abs{v}=n}\P\bigg([X]_v^{oU(v)}>\tfrac{1}{2rL(v)}\bigg)
		\leq\sum_{\abs{v}=n}\P\bigg(\abs{[X]_v^{oU(v)}}>\tfrac{1}{2rL(v)}\bigg)\\
		\nonumber&\leq\sum_{\abs{v}=n}\P\bigg(\abs{\scalar{oU(v)e^{}_3,[X]_v}}>\tfrac{1}{2rL(v)}\bigg)
		\leq\sum_{\abs{v}=n}\P(\norm{oU(v)e^{}_3}\norm{ X}>\tfrac{1}{2rL(v)})\\&=\sum_{\abs{v}=n}\P(\norm{X}>\tfrac{1}{2rL(v)})
	\end{align}
	By assumption \eqref{K:assumption_regular_variation}
	\begin{equation}\label{K:proof_applying_regular_variation}
		\sum_{\abs{v}=n}\P(\norm{X}>\tfrac{1}{2rL(v)})\leq\sum_{\abs{v}=n} 2^\alpha  Cr^\alpha L(v)^\alpha.
	\end{equation}
	By taking the limit in \eqref{K:proof_applying_regular_variation} as $n\rightarrow\infty$ and combining the obtained result with \eqref{K:proof_rewriting_vague_convergence} and \eqref{K:sum_of_tails}, we obtain
	\begin{equation*}
		\nu^o((\tfrac{1}{r},\infty))\leq 2^\alpha C r^\alpha W_\infty\quad\as
	\end{equation*}
In the same way (observe that we have taken absolute values in \eqref{K:sum_of_tails} already), we obtain	\begin{equation*}
		\nu^o((-\infty,-\tfrac{1}{r}))\leq C r^\alpha W_\infty\quad\as
	\end{equation*}
\end{proof}

With the help of Lemma above, we are now able to establish an important invariance property of $\nu^{o}$. Recall that $\Support$ denotes the smallest closed subgroup of $\Sim$  that contains $\{(L(v),U(v))_{v \in \V}\}$ with probability 1. We say that a family $(\bar{\nu}^o)_{o \in \OrthogonalGroup}$ of deterministic L\'evy measures on $\R$ is $(\Support,\alpha)$-invariant, if for all $o \in \OrthogonalGroup$ it holds that	\begin{equation}\label{kin:formula_characteristics_levy_triplet_levy_measure_right-hand_side}
	\bar{\nu}^{o}(s^{-1}B)=s^{\alpha}\bar{\nu}^{ou^{-1}}(B)
\end{equation} 
for any $B\subset \Rnot$ and $(s,u)\in\Support$. Equivalently,
\begin{equation}\label{kin:formula_characteristics_levy_triplet_levy_measure_right-hand_side-002}
	\int f(sy) \bar{\nu}^o(\d y)=s^\alpha \int f(y) \bar{\nu}^{ou^{-1}}(\d y)
\end{equation}
for every nonnegative measurable function $f$ on $\Rnot$.

\begin{lem}\label{K:proposition_levy_triplet_nu_characteristics}
	 A family of random L\'evy measures $(\nu^{o})_{o \in \OrthogonalGroup}$ satisfies \eqref{K:levy_triplet_nu_branching_property} if and only if there is a family of deterministic $(\Support,\alpha)$-invariant L\'evy measures $(\bar{\nu}^{o})_{o \in \OrthogonalGroup}$ such that  
	 \begin{equation}\label{K:proof_proposition_levy_triplet_nu_characteristics_deterministic_and_random_nu_connection}
	 	\nu^{o}=W_\infty\bar{\nu}^{o}
	 \end{equation}
almost surely for all $o \in \OrthogonalGroup$. There is a constant $C$ such that for all $r>0$ and $o \in \OrthogonalGroup$,
	\begin{equation*}
	\bar{\nu}^o((\tfrac{1}{r},\infty))\leq C r^\alpha  \quad \text{ and } 	\quad \bar{\nu}^o((-\infty,-\tfrac{1}{r}))\leq C r^\alpha .
\end{equation*}
\end{lem}
\begin{proof} 
	The proof is organized as follows. In step 1 we prove the sufficiency by showing, that $\nu^o$ given by \eqref{K:proof_proposition_levy_triplet_nu_characteristics_deterministic_and_random_nu_connection} satisfies \eqref{K:levy_triplet_nu_branching_property}. The general idea to prove the necessity, which is the more complicated part of the two, is to define a deterministic measure $\bar{\nu}^o:=\E[\nu^o]$ and show, that it satisfies \eqref{K:proof_proposition_levy_triplet_nu_characteristics_deterministic_and_random_nu_connection} and \eqref{kin:formula_characteristics_levy_triplet_levy_measure_right-hand_side} on a set of intervals forming a countable generator of the Borel-$\sigma$-field on $\Rnot$. This will immediately imply, that both these properties hold for any subset of $\Rnot$. 
	In step 2 we introduce the countable generator $S$ of Borel-$\sigma$-field on $\Rnot$ and show, that the map $(r,o)\mapsto \E[ r^{-\alpha}\nu^o((\frac1r,\infty))]$ is constant $\rho$-a.s., where $\rho$ is defined in Lemma \ref{K:many-to-one-lemma}. In step 3 we define $\bar{\nu}^o$ and prove \eqref{K:proof_proposition_levy_triplet_nu_characteristics_deterministic_and_random_nu_connection} and, consequently, \eqref{kin:formula_characteristics_levy_triplet_levy_measure_right-hand_side}. In step 4 we show, that $\bar{\nu}^o$ is a L\'evy measure, thus finishing the proof.
	
	\textbf{Step 1}: 
	Suppose that 	$\nu^{o}=W_\infty\bar{\nu}^{o}$ for a family $(\bar{\nu}_o)_{o \in \OrthogonalGroup}$ satisfying \eqref{kin:formula_characteristics_levy_triplet_levy_measure_right-hand_side}.
	 For any $B\subset\Rnot$, in view of \eqref{K:martingale_W_n_limit_branching_property}, we then have
	\begin{align*}
		&\sum_{\abs{v}=n}[\nu]_v^{oU(v)}(L(v)^{-1}B)=\sum_{\abs{v}=n}[W_\infty\bar{\nu}]_v^{oU(v)}(L(v)^{-1}B)\\ 
		&= \sum_{\abs{v}=n}[W_\infty]_v\bar{\nu}^{oU(v)}(L(v)^{-1}B)
		=\sum_{\abs{v}=n}[W_\infty]_vL(v)^{\alpha}\bar{\nu}^{oU(v)U(v)^{-1}}(B)\\
		&=W_\infty\bar{\nu}^{o}(B)=\nu^{o}(B),\quad \as
	\end{align*}
	
	In the rest of the proof we prove the necessity.
	
	\textbf{Step 2}:
	Assume that the family $(\nu^o)_{o \in \OrthogonalGroup}$ satisfies the equation \eqref{K:levy_triplet_nu_branching_property}  for any measurable set $B\subset\Rnot$. 
	For any $x>0$ denote by $I_x$ and $I_{-x}$ intervals $(\frac{1}{x},\infty)$ and $(-\infty,-\frac{1}{x})$. Note, that the collection of sets $S:=\big\{ I_{r} \, :\, r\in\mathbb{Q}_{> } \big\} \cup \big\{ I_{-r} \, : \, r\in\mathbb{Q}_{< }\big\}$ is a countable generator of the Borel-$\sigma$-field on $\Rnot$. For any $r>0$ and $o\in\OrthogonalGroup$ we define the following functionals.
	\begin{align}\label{kin:proof_functions_phi_and_psi}
		\nonumber\psi_{+}(r,o):=r^{-\alpha}\nu^{o}(I_r),\quad \eta_{+}(r,o):=\E[\psi_{+}(r,o)],\\ \psi_{-}(r,o):=r^{-\alpha}\nu^{o}(I_{-r}),\quad \eta_{-}(r,o):=\E[\psi_{-}(r,o)].
	\end{align}
	Observe, that the integrability of $\psi_\pm$ is a consequence of Lemma \ref{K:levy_triplet_nu_boundedness_lemma}. Using $\E W_\infty=1$, it holds moreover  that $\eta_\pm(r,o) \le C$ for all $r>0$, $o \in \OrthogonalGroup$ for the constant $C$ from Lemma \ref{K:levy_triplet_nu_boundedness_lemma}.
	By \eqref{K:levy_triplet_nu_branching_property} we have
	\begin{align*}
		\psi_+(r,o)=r^{-\alpha}\sum_{\abs{v}=n}[\nu]_v^{oU(v)}(L(v)^{-1}I_r)&\\
		=r^{-\alpha}\sum_{\abs{v}=n}[\nu]_v^{oU(v)}(I_{rL(v)})=&\sum_{\abs{v}=n}L(v)^{\alpha}[\psi_+]_v(rL(v),oU(v)),\quad \as
	\end{align*}
	By taking the expectation and using the many-to-one identity, Lemma \ref{K:many-to-one-lemma}, we obtain
	\begin{align}\label{K:proof_proposition_levy_triplet_nu_characteristics_formula_1}
		\nonumber\eta_+(r,o)&=\E\bigg[\sum_{\abs{v}=n}L(v)^{\alpha}[\psi_+]_v(rL(v),oU(v))\bigg]
		\\
		&=\E\bigg[\sum_{\abs{v}=n}L(v)^{\alpha}\eta_+(rL(v),oU(v))\bigg]=\E\bigg[\eta_+(r\sL_n,o\sU_n)\bigg],
	\end{align}
	where $(\sL_n,\sU_n)_{n\in\N_0}$ is the random walk on $\Sim$, introduced in Lemma \ref{K:many-to-one-lemma}.
	Defining for any fixed $(r,o)\in\Sim$ a function $F_{r,o}:\Support\rightarrow[0,\infty)$ as
	\begin{equation*}
		F_{r,o}(s,u):=\eta_{+}(rs,ou),
	\end{equation*}
	the property \eqref{K:proof_proposition_levy_triplet_nu_characteristics_formula_1} (applied at $r'=rs$, $o'=ou$) translates into 
	\begin{equation*}
		F_{r,o}(s,u)=\E[ F_{r,o}(s\sL_n,u\sU_n)],
	\end{equation*} 
	for any $(s,u)\in\Support$. 
	Thus, $F_{r,o}$ satisfies \eqref{K:Choquet-Deny_functional_equation}, moreover, it is bounded, since $\eta_+$ is bounded. Applying the Choquet-Deny Lemma \ref{K:Choquet-Deny-lemma} to $F_{r,o}$, we infer, that $F_{r,o}$ is constant $\rho$-a.s., were $\rho$ is the increment distribution in the associated random walk $(\sL_n,\sU_n)_{n\in\N_0}$, introduced in the many-to-one lemma (Lemma  \ref{K:many-to-one-lemma}).
	It follows that for any $r>0$, $o \in \OrthogonalGroup$,
	\begin{equation}\label{eq:invariance_eta}
		 \eta_+(r,o)=\eta_+(rs, ou)
	\end{equation}
	for $\rho$-almost every $(s,u) \in \Support$.
	
		\textbf{Step 3}: Let $H^c$ be a set in $\SpaceRO$, such that $\rho(H^c)=\P((\sL_1,\sU_1)\in H^c)=1$ and $\eta_+(r,o)$ is constant on $H^c$. Applying Lemma \ref{K:Choquet-Deny-lemma} to the function $\1_{H^c}(\cdot)$ gives us
	\begin{equation*}
		\E\bigg[\sum_{\abs{v}=n}L(v)^\alpha \1_{H^c}(L(v),U(v))\bigg]=\E[\1_{H^c}(\sL_n,\sU_n)]=\rho(H^c)=1,
	\end{equation*}
	hence, by assumption \eqref{K:assumption_m},
	\begin{equation*}
		\E\bigg[\sum_{\abs{v}=n}L(v)^\alpha \1_{H^c}(L(v),U(v))\bigg]=\E\bigg[\sum_{\abs{v}=n}L(v)^\alpha\bigg].
	\end{equation*}
	The last equality implies 
	\begin{equation}\label{K:proof_proposition_levy_triplet_nu_characteristics_formula_2}
		\P((L(v),U(v))\in H^c~\text{for all}~v:L(v)>0)=1,
	\end{equation}
	for all $n\in\N$.
	By martingale convergence theorem, in view of  \eqref{K:proof_proposition_levy_triplet_nu_characteristics_formula_2} we have
	\begin{align*}				 			  \psi_+(r,o)&=\lim_{n\rightarrow\infty}\E\bigg[\psi_+(r,o)|F_n\bigg]=\lim_{n\rightarrow\infty}\E\bigg[\sum_{\abs{v}=n}L(v)^{\alpha}[\psi_+]_v(rL(v),oU(v))|F_n\bigg]\\
		&=\lim_{n\rightarrow\infty}\sum_{\abs{v}=n}L(v)^{\alpha}\eta_+(rL(v),oU(v))=W_\infty\eta_+(r,o),\quad\as,	
	\end{align*} simultaneously for all $(r,o)$.  
	Recalling the definition of $\psi_+$ in \eqref{kin:proof_functions_phi_and_psi}, we obtain 
	\begin{equation}\label{K:proof_representation_of_nu}
		\nu^{o}(I_r)=W_\infty\eta_+(r,o)r^{\alpha}, \quad\as
	\end{equation}
	We now define by
	\begin{equation}\label{K:levy_triplet_levy_measure_deterministic}
		\bar{\nu}^{o}:=\E[\nu^{o}]
	\end{equation}
	a deterministic measure on $\Rnot$. By \eqref{K:proof_representation_of_nu}, we have
	\begin{equation*}
		\bar{\nu}^{o}(I_r)=\E[\nu^{o}(I_r)]=\E[ W_\infty\eta_+(r,o)r^{\alpha}]=\eta_+(r,o)r^{\alpha},
	\end{equation*}
	which, again in view of \eqref{K:proof_representation_of_nu}, implies
	\begin{equation}\label{K:proof_levy_measure_connect}
		\nu^o(I_r)=W_\infty\bar{\nu}^{o}(I_r),\quad \as
	\end{equation}
	Further, for any $(s,u)\in\Support$ it follows by \eqref{eq:invariance_eta}, that
	\begin{align}\label{kin:proof_nu_property}
		\nu^{o}(s^{-1}I_r)&=\nu^{o}(I_{rs})=W_\infty\eta_+(rs,o)r^{\alpha}s^{\alpha}
		=W_\infty\eta_+(rs,(ou^{-1})u)r^{\alpha}s^{\alpha}\nonumber\\
		&=s^{\alpha}W_\infty\eta_+(r,ou^{-1})r^{\alpha}=s^{\alpha}\nu^{ou^{-1}}(I_r),\quad \as
	\end{align}
	Formulae \eqref{K:proof_levy_measure_connect} and \eqref{kin:proof_nu_property} imply, that relations \eqref{K:proof_proposition_levy_triplet_nu_characteristics_deterministic_and_random_nu_connection} and \eqref{kin:formula_characteristics_levy_triplet_levy_measure_right-hand_side}, respectively, hold for all sets $(I_r)_{r\in\mathbb{Q}_>}$. Replacing $\psi_+$ and $\eta_+$ with $\psi_-$ and $\eta_-$ in the above proof, by same arguments both these properties hold for all sets $(I_{r})_{r\in\mathbb{Q}_<}$, hence for all sets in $S$.  Since $S$ is a countable generator of the Borel-$\sigma$-field on $\Rnot$, we conclude, that \eqref{K:proof_proposition_levy_triplet_nu_characteristics_deterministic_and_random_nu_connection} and \eqref{kin:formula_characteristics_levy_triplet_levy_measure_right-hand_side} hold for all measurable sets in $\Rnot$. 
	
	\textbf{Step 4}:
	We now show, that $\bar{\nu}^o$ is a (deterministic) L\'evy measure on $\R$, for any $o\in\OrthogonalGroup$. Since $\nu^o$ itself is a L\'evy measure, by \eqref{K:proof_proposition_levy_triplet_nu_characteristics_deterministic_and_random_nu_connection} we trivially have $\bar{\nu}^o(\{0\})=0$ and further observe, that
	\begin{equation*}
		\infty > \int (x^2\wedge 1){\nu}^o(\d x)=\int (x^2\wedge 1)W_\infty\bar{\nu}^o(\d x), \quad\text{a.s.}
	\end{equation*}
	W.l.o.g. assume, that the above relation holds on a set $H^c$, such that $\P(\mathbf{L}\in H^c)=1$. Since the set $\{W_\infty>0\}$ carries positive probability, we can always find an $\mathbf{l}\in H^c \cap \{W_\infty>0\}$, such that
	\begin{equation}
		\int (x^2\wedge 1)W_\infty(\mathbf{l})\bar{\nu}^o(\d x)<\infty
	\end{equation} 
	holds. Dividing the last inequality by $W_\infty(\mathbf{l})$ completes the proof.
\end{proof}

\subsection{Determining the random L\'evy-triplet}

Lemma \ref{K:proposition_levy_triplet_nu_characteristics} is a powerful tool, which enables us to further study L\'evy-triplet $(\mu^o,\Theta^o,\nu^o)$. In the next lemma we derive the more convenient form for the integral in the characteristic exponent $\Psi^o$.

Given a family $(\bar{\nu}^o)_{o \in \OrthogonalGroup}$ of $(\Support,\alpha)$-invariant L\'evy measures, define
\begin{align}\label{K:levy_triplet_integral_constants}
	\nonumber &{K}_1(r,o):=\int_{}(\cos(ry)-1)\bar{\nu}^{o}(\d y),\\
	\nonumber &{K}_2(r,o):=\int_{}\sin(ry)\bar{\nu}^{o}(\d y),\\
	&{\gamma}(o):=-\int_{}y\1_{[0,1]}(|y|)\bar{\nu}^{o}(\d y).
\end{align}
Further, denote by 	\begin{equation*}
	\I^o(x):=\int_{} \bigg(e^{ix y}-1-ix y\1_{[0,1]}(y)\bigg)\nu^{o}(\d y)
\end{equation*}
the L\'evy integral that appears  in the random L\'evy-Khintchine exponent $\Psi^o$ in formula \eqref{K:characteristic_exponent_Psi}.

\begin{lem}\label{K:proposition_levy_triplet_integral_characteristics}
	With the notations \eqref{K:levy_triplet_integral_constants}, it holds for any $r\in\Rplus$ and $o\in\OrthogonalGroup$ that
	\begin{equation}\label{K:characteristic_exponent_integral_decomposition}
		\I^o(r)=W_\infty{K}_1(r,o) +iW_\infty{K}_2(r,o) + irW_\infty{\gamma}(o).
	\end{equation}
	The functions $K_j$, $j=1,2$ are $(\Support,\alpha)$-invariant.
\end{lem}
\begin{proof}
	Fix any $o\in\OrthogonalGroup$. For any $r\in\Rplus$ we have
	\begin{equation}
		\I^o(r)=\int_{} \bigg(\cos(ry) + i\sin(ry) -1-ir y\1_{[0,1]}(y)\bigg)W_\infty\bar{\nu}^{o}(\d y),
	\end{equation}
	which, in view of notations in \eqref{K:levy_triplet_integral_constants}, yields \eqref{K:characteristic_exponent_integral_decomposition}. Further, by Proposition \ref{K:proposition_levy_triplet_nu_characteristics}, in particular the invariance property \eqref{kin:formula_characteristics_levy_triplet_levy_measure_right-hand_side-002}, for any $(s,u)\in\Support$ we have
	\begin{equation*}
		{K}_2(s,o)=\int_{}\sin(sy)\bar{\nu}^{o}(\d y)=
		s^{\alpha}\int_{}\sin(y)\bar{\nu}^{ou^{-1}}(\d y)=s^{\alpha}{K}_2(1,ou^{-1})
	\end{equation*}
	and
	\begin{equation*}
		{K}_1(s,o)=\int_{}(\cos(sy)-1)\bar{\nu}^{o}(\d y)=s^{\alpha}\int_{}(\cos(y)-1)\bar{\nu}^{ou^{-1}}(\d y)=s^{\alpha}{K}_1(1,ou^{-1}).
	\end{equation*}
\end{proof}

We proceed to show that, depending on the value of $\alpha$, certain components in the random L\'evy-triplet may vanish.
 We start with uniform boundedness of $\gamma$ in case $\alpha\in(0,1)$, from which we will infer that both $\mu^{o}$ and $\Theta^{o}$ vanish for $\alpha \in (0,1)$.
 
\begin{lem}\label{K:proposition_uniform_boundedness_of_gamma_o}
	Assume $\alpha\in(0,1)$. Then the function $\gamma$, defined in \eqref{K:levy_triplet_integral_constants}, is bounded uniformly in $\OrthogonalGroup$.
\end{lem}
\begin{proof}	
	Observe that
	\begin{align}\label{K:proof_proposition_levy_triplet_integral_characteristics_gamma_two_modules}
		\abs{{\gamma}(o)}=\bigg|\int_{}y\1_{[0,1]}(|y|)\bar{\nu}^{o}(\d y)\bigg|\leq \bigg|\int_{(0,1]}y\bar{\nu}^{o}(\d y)\bigg|+\bigg|\int_{[-1,0)}y\bar{\nu}^{o}(\d y)\bigg|.
	\end{align}
	We will treat the two summands in the right-hand side of  \eqref{K:proof_proposition_levy_triplet_integral_characteristics_gamma_two_modules} separately. We use \eqref{kin:formula_characteristics_levy_triplet_levy_measure_right-hand_side} and the fact, that $\bar{\nu}^{o}(I_s)$ is finite for any $s\in\Rplus$ and $o\in\OrthogonalGroup$, which is a direct consequence of Lemma \ref{K:levy_triplet_nu_boundedness_lemma}. As a consequence of $\eqref{K:assumption_m}$, we may assume that there exists $s\neq1$ and $u\in\OrthogonalGroup$, such that $(s,u)\in\Support$. By the group property of $\Support$ we may further assume that $s>1$. Direct calculations, using \eqref{kin:formula_characteristics_levy_triplet_levy_measure_right-hand_side-002}, yield
	\begin{align}\label{K:proof_proposition_levy_triplet_integral_characteristics_formula_first}
		\nonumber\bigg|\int_{(0,1]}y\bar{\nu}^{o}(\d y)\bigg|
		\nonumber&\leq\sum_{n\ge 0}\bigg|\int \mathds{1}_{(s^{-n-1},s^{-n}]}(y)\,  y\bar{\nu}^{o}(\d y)\bigg| = \sum_{n\ge 0}\bigg|\int \mathds{1}_{(s^{-1},1]}(s^ny)\,  s^{-n} (s^ny)\bar{\nu}^{o}(\d y)\bigg|
		\\
		\nonumber&=\sum_{n\ge 0}\bigg|s^{n\alpha}\int_{(\frac{1}{s},1]} s^{-n}y\bar{\nu}^{o(u)^{-n}}(\d y)\bigg|
		=\sum_{n\ge 0}s^{n(\alpha-1)}\bigg|\int_{(\frac{1}{s},1]} y\bar{\nu}^{o(u)^{-n}}(\d y)\bigg|
		\\
		\nonumber&\leq\sum_{n\ge 0}s^{n(\alpha-1)}\bigg|\int_{(\frac{1}{s},1]}\bar{\nu}^{o(u)^{-n}}(\d y)\bigg|
		\leq\sum_{n\ge 0}s^{n(\alpha-1)}\abs{\bar{\nu}^{o(u)^{-n}}((\tfrac{1}{s},\infty)}
		\\
		&\leq\sum_{n\ge 0}(s^{\alpha-1})^nCs^{\alpha}=:\widetilde{C}<\infty,
	\end{align}	
	where we have used Lemma \ref{K:proposition_levy_triplet_nu_characteristics} in the penultimate line. The constant $\widetilde{C}$ is finite due to $s>1$ and $\alpha<1$. The second summand in \eqref{K:proof_proposition_levy_triplet_integral_characteristics_gamma_two_modules} is treated analogously.
\end{proof}

We proceed by proving the uniform boundedness of $\mu^o$ in case $\alpha\in(0,1)$.
\begin{lem}\label{K:proposition_uniform_boundedness_of_mu_o}
	 The random shift $\mu^o$ in the L\'evy-triplet is an $F_\infty$-measurable function. If  $\alpha\in(0,1)$, then 
	 $$ \sup_{o \in \OrthogonalGroup} |\mu^o| < \infty \quad \as$$
\end{lem}
\begin{proof}
	Choose any $h\in(0,1)$, such that $\bar{\nu}^o(\{\abs{x}=h\})=0$ and hence, by Lemma \ref{K:proposition_levy_triplet_nu_characteristics}, ${\nu}^o(\{\abs{x}=h\})=0$ almost surely.  By part (iii) in Theorem 13.28 in \cite{Kallenberg:1997}, we have 
	\begin{equation}\label{eq:boh as limit}
		\sum_{\abs{v}=n}\E\bigg[L(v)[X]_v^{oU(v)};~\abs{L(v)[X]_v^{oU(v)}}\leq h\bigg]\rightarrow b^o(h) \quad \text{ a.s.},
	\end{equation}
	where $b^o(h)$ is defined as
	\begin{equation}\label{K:Levy-triplet_b^o(h)}
		b^o(h):=\mu^o-\int_{h<\abs{x}\leq 1}x\nu^o(\d x). 
	\end{equation}
	Both equations together yield in particular the measurability of first $b^o(h)$ and hence of $\mu^o$. By an appeal to Proposition \ref{prop:SimultaneousConvergence}, it further holds that the family $\{ \mu_o \, : \, o \in \OrthogonalGroup\}$ is well defined simultaneously outside the set $N$ of measure zero.
	
	Concerning the boundedness for $\alpha<1$, it  follows from \eqref{eq:boh as limit}, that
	\begin{equation}\label{K:proof_proposition_uniform_boundedness_of_mu_o_convergence_of_modules}
		\biggabs{\sum_{\abs{v}=n}\E\bigg[L(v)[X]_v^{oU(v)};~\abs{L(v)[X]_v^{oU(v)}}\leq h\bigg]}\rightarrow \abs{b^o(h)} \quad \text{ a.s.}.
	\end{equation}
	For the left side of \eqref{K:proof_proposition_uniform_boundedness_of_mu_o_convergence_of_modules}, we write
	\begin{align*}
		&\biggabs{\sum_{\abs{v}=n}\E\bigg[L(v)[X]_v^{oU(v)};~\abs{L(v)[X]_v^{oU(v)}}\leq h\bigg]}
		\leq\sum_{\abs{v}=n}\E\bigg[\abs{L(v)[X]_v^{oU(v)}}\1_{[0,h]}(\abs{L(v)[X]_v^{oU(v)}})\bigg]
		\\
		&\le\sum_{\abs{v}=n}\int_{0}^{h}\P\bigg(\abs{L(v)[X]_v^{oU(v)}}>x\bigg)\d x
		=\sum_{\abs{v}=n}\int_{0}^{h}\P\bigg(\abs{\scalar{oU(v)e^{}_3,X}}>\tfrac{x}{L(v)}\bigg)\d x
		\\
		&\leq\sum_{\abs{v}=n}\int_{0}^{h}\P\bigg(\norm{oU(v)e^{}_3}\norm{X}>\tfrac{x}{L(v)}\bigg)\d x
		\leq\sum_{\abs{v}=n}\int_{0}^{h}\P\bigg(\norm{X}>\tfrac{x}{L(v)}\bigg)\d x
		\\
		&\leq \sum_{\abs{v}=n}\int_{0}^{h}CL(v)^{\alpha}x^{-\alpha}\d x=\bigg(\sum_{\abs{v}=n}L(v)^{\alpha}\bigg)C\frac{h^{1-\alpha}}{1-\alpha},
	\end{align*}
	where the last inequality holds by assumption \eqref{K:assumption_regular_variation}. Sending $n\rightarrow\infty$, in view of \eqref{K:martingale_W_n_limit_branching_property}, we obtain
	\begin{equation*}
		\lim_{n\rightarrow\infty}\biggabs{\sum_{\abs{v}=n}\E\bigg[L(v)[X]_v^{oU(v)};~\abs{L(v)[X]_v^{oU(v)}}\leq h\bigg]}\leq W_\infty C\frac{h^{1-\alpha}}{1-\alpha} \quad \text{a.s.},
	\end{equation*}
	which, combined with \eqref{K:proof_proposition_uniform_boundedness_of_mu_o_convergence_of_modules}, yields
	\begin{equation}\label{K:proof_proposition_uniform_boundedness_of_mu_o_boundedness_1}
		\abs{b^o(h)}\leq W_\infty C\frac{h^{1-\alpha}}{1-\alpha} \quad \text{a.s.}
	\end{equation}
	for any $h>0$. For the integral in \eqref{K:Levy-triplet_b^o(h)} we have
	\begin{equation}\label{K:proof_proposition_uniform_boundedness_of_mu_o_boundedness_2}
		\biggabs{\int_{h<\abs{x}\leq 1}x\nu^o(\d x)}\leq \abs{W_\infty{\gamma}(o)}\leq W_\infty\widetilde{C} \quad \text{a.s.},
	\end{equation}
	by Lemma \ref{K:proposition_levy_triplet_integral_characteristics}. Combining \eqref{K:proof_proposition_uniform_boundedness_of_mu_o_boundedness_1} and \eqref{K:proof_proposition_uniform_boundedness_of_mu_o_boundedness_2}, we obtain
	\begin{equation*}
		\abs{\mu^o}\leq \abs{b^o(h)} + \biggabs{\int_{h<\abs{x}\leq 1}x\nu^o(\d x)}\leq \abs{W_\infty{\gamma}(o)}\leq W_\infty C\frac{h^{1-\alpha}}{1-\alpha}+W_\infty \widetilde{C} \quad \text{a.s.}.
	\end{equation*}
	Since both constants do not depend on $o$ and are finite, this implies the uniform boundedness of $\mu^o$ almost surely in $\OrthogonalGroup$.
\end{proof}

We now prove the uniform boundedness of $\Theta^o$ for both $\alpha\in(0,1)$ and $\alpha\in(1,2)$. 
\begin{lem}\label{K:proposition_uniform_boundedness_of_theta_o}
	 The random variance $\Theta^o$  is an $F_\infty$-measurable function. For $\alpha\in(0,1)\cup(1,2)$, it holds
	 $$ \sup_{o \in \OrthogonalGroup} |\Theta^o| < \infty \quad \as.$$
\end{lem}
\begin{proof}
	As in the previous lemma, we choose $h>0$, such that $\bar{\nu}^o(\{\abs{x}=h\})=0$ and hence, by Lemma \ref{K:proposition_levy_triplet_nu_characteristics}, ${\nu}^o(\{\abs{x}=h\})=0$ a.s.  By part (ii) in Theorem 13.28 in \cite{Kallenberg:1997}, 
	\begin{equation}\label{K:proof_proposition_uniform_boundedness_of_theta_o_kallenberg}
		\sum_{\abs{v}=n}\V\bigg[L(v)[X]_v^{oU(v)};~\abs{L(v)[X]_v^{oU(v)}}\leq h\bigg]\rightarrow a^o(h) \quad \text{a.s.},
	\end{equation}
	as $n\rightarrow\infty$, where $a^o(h)$ is defined as
	\begin{equation}\label{K:Levy-triplet_a^o(h)}
		a^o(h):=\Theta^o+\int_{h<\abs{x}\leq 1}x^2\nu^o(\d x).
	\end{equation}
	This proves in particular the $F_\infty$-measurability of $\Theta^o$.

	By Lemma \ref{K:proposition_levy_triplet_nu_characteristics}, the integral in \eqref{K:Levy-triplet_a^o(h)} is bounded by $C h^{-\alpha}$ uniformly in $\OrthogonalGroup$.
	Boundedness of $\Theta^{o}$ will hence be deduced from boundedness of $a^o(h)$. Therefore, we want to bound the left hand side of \eqref{K:proof_proposition_uniform_boundedness_of_theta_o_kallenberg} from above. Obviously, we can bound variances by the second moment. Hence we consider
	\begin{align*}
		&\sum_{\abs{v}=n}\E\bigg[(L(v)[X]_v^{oU(v)}\1_{[0,h]}(\abs{L(v)[X]_v^{oU(v)}}))^2\bigg]
		\\
		&\leq\sum_{\abs{v}=n}\int_{0}^{h^2}\P\bigg((L(v)[X]_v^{oU(v)})^2>x\bigg)\d x
		\\
		&\leq\sum_{\abs{v}=n}\int_{0}^{h^2}\P\bigg(([X]_v^{oU(v)})^2>\tfrac{x}{L(v)^{2}}\bigg)\d x
		=\sum_{\abs{v}=n}\int_{0}^{h^2}\P\bigg(\abs{[X]_v^{oU(v)}}>\tfrac{\sqrt{x}}{L(v)}\bigg)\d x
		\\
		&=\sum_{\abs{v}=n}\int_{0}^{h^2}\P\bigg(\abs{\scalar{oU(v)e^{}_3,X}}>\tfrac{\sqrt{x}}{L(v)}\bigg)\d x
		\leq\sum_{\abs{v}=n}\int_{0}^{h^2}\P\bigg(\norm{X}>\tfrac{\sqrt{x}}{L(v)}\bigg)\d x
		\\
		&\leq \sum_{\abs{v}=n}\int_{0}^{h^2}CL(v)^{\alpha}x^{-\frac{\alpha}{2}}\d x=\bigg(\sum_{\abs{v}=n}L(v)^{\alpha}\bigg)C\frac{h^{2-\alpha}}{1-\frac{\alpha}{2}},
	\end{align*}
	where the last inequality holds by assumption \eqref{K:assumption_regular_variation}. Sending $n\rightarrow\infty$, in view of \eqref{K:martingale_W_n_limit_branching_property}, we obtain
	\begin{equation*}
		\lim_{n\rightarrow\infty}\sum_{\abs{v}=n}\E\bigg[(L(v)[X]_v^{oU(v)}\1_{[0,h]}(\abs{L(v)[X]_v^{oU(v)}}))^2\bigg]\leq W_\infty C\frac{h^{2-\alpha}}{1-\frac{\alpha}{2}},
	\end{equation*}
	which yields
	\begin{equation}\label{K:proof_proposition_uniform_boundedness_of_theta_o_boundedness_1}
		a^o(h)\leq W_\infty C\frac{h^{1-\alpha}}{1-\alpha}.
	\end{equation}
	Together with \eqref{K:Levy-triplet_a^o(h)}, this immediately implies that $\Theta^o$ is almost surely uniformly bounded in $\OrthogonalGroup$.
\end{proof} 

\subsection{Proof of Theorem \ref{K:proposition_main_result}}\label{subsection:proofMainresult}
We now have all ingredients to prove Theorem \ref{K:proposition_main_result}. We start by proving the converse inclusion, that is, every solution is of the form \eqref{eq:structure of solutions}. 

A short proof for the direct inclusion  follows immediately afterwards.

\begin{proof}[Proof of the converse inclusion in Theorem \ref{K:proposition_main_result}]
 Let $\phi$ be the characteristic function of a random vector on $\R^3$ that is a solution to \eqref{K:stationary_equation}. Recall that we consider $\alpha>0$ with $\alpha \neq 1$. By Remark \ref{rem:represenation:phi:psi}, we have for any $r\in\Rplus$ and $o\in\OrthogonalGroup$
	\begin{equation*}
		\phi(roe_3)=\E[ e^{\Psi^o(r)}].
	\end{equation*}
	By Lemma \ref{K:proposition_levy_triplet_integral_characteristics},  we have
	\begin{equation}\label{K:proof_main_result_psi_reminder}
		\Psi^{o}(r)=ir\mu^{o}-\frac{1}{2}r^2\Theta^{o}+W_\infty{K}_1(r,o)+iW_\infty{K}_2(r,o)+irW_\infty{\gamma}(o).
	\end{equation}
Using the invariance formulae obtained, for any $(s,u) \in \Support$ we have
	\begin{equation}\label{K:proof_main_result_psi_expanded}
		\Psi^o(rs)=irs\mu^o-\frac{1}{2}r^2s^2\Theta^{o}+s^{\alpha}W_\infty {K}_1(r,ou^{-1})+is^{\alpha}W_\infty{K}_2(r,ou^{-1})+irsW_\infty{\gamma}(o).
	\end{equation}
	The proof is organized as follows. We first prove the representation \eqref{eq:structure of solutions} by using the branching relation \eqref{K:characterisric_exponent_Psi_branching_property} satisfied by $\Psi^{o}$ and separating the components of $\Psi^{o}$ to obtain equations for $Y$, $V$ and $K$, to be defined below as well.
	Then we turn to the separate components and prove that they vanish if $\alpha$ takes values in specified ranges.

	\textbf{Step 1}.
	In this step, we are going to prove for any $\alpha \neq 1$ the representation  \eqref{eq:structure of solutions} and show, that $K$ vanishes a.s. for $\alpha=2$.

By \eqref{K:characterisric_exponent_Psi_branching_property} we have for any $r>0$, $s \in \G$ and $o\in\OrthogonalGroup$ that
	\begin{equation}\label{K:proof_main_result_psi_branching_reminder}
		\Psi^o(rs)=\sum_{\abs{v}=n}[\Psi]_v^{oU(v)}(rsL(v)),\quad\as
	\end{equation} 
	Fix any $s\in\G$, such that $(s,u)\in\Support$ for some $u\in\U$. Expanding the right-hand side in the above formula, using \eqref{K:proof_main_result_psi_reminder} and \eqref{K:proof_main_result_psi_expanded} we obtain
	\begin{align}
		\sum_{\abs{v}=n}[\Psi]_v^{oU(v)}(rsL(v))=\sum_{\abs{v}=n}&\bigg(irs[\mu]_v^{oU(v)}-\frac{1}{2}r^2s^2L(v)^2[\Theta]^{oU(v)}_v+[W_\infty]_v{K}_1(rsL(v),oU(v)) \notag
		\\
		&+i[W_\infty]_v{K}_2(rsL(v),oU(v))+i[W_\infty]_vrsL(v){\gamma}(oU(v))\bigg) \notag
		\\
		=\sum_{\abs{v}=n}&\bigg(irs[\mu]_v^{oU(v)}-\frac{1}{2}r^2s^2L(v)^2[\Theta]^{oU(v)}_v+[W_\infty]_vL(v)^{\alpha}s^{\alpha}{K}_1(r,ou^{-1}) \notag
		\\
		&+i[W_\infty]_vL(v)^{\alpha}s^{\alpha}{K}_2(r,ou^{-1})+i[W_\infty]_vrsL(v){\gamma}(oU(v))\bigg) \label{K:proof_main_result_sum_psi_expanded}.
	\end{align}
	In the last line, we have used the $(\Support,\alpha)$-invariance of $K_1$ and $K_2$, to take $L(v)^\alpha$ outside. At the same time, this removes $U(v)$.
	Let us define
		\begin{equation}\label{K:proof_main_result_defining_Y}
		Y(o):=\mu^{o}+W_\infty{\gamma}(o)\quad\text{and}\quad V(o):=\Theta^o.
	\end{equation}
 Using these definitions in \eqref{K:proof_main_result_psi_expanded} and \eqref{K:proof_main_result_sum_psi_expanded}, we can rewrite the identity \eqref{K:proof_main_result_psi_branching_reminder} as
	\begin{align}
			&irsY(o)-\frac{1}{2}r^2s^2V(o) + s^{\alpha}W_\infty K_1(r,ou^{-1}) + is^{\alpha}W_\infty K_2(r,ou^{-1}) \notag \\
	=~&irs\sum_{|v|=n}L(v)[Y]_v(oU(v))-\frac{1}{2}r^2s^2\sum_{|v|=n}L(v)^2[V]_v(oU(v)) \notag \\
		&+ s^{\alpha}\bigg(K_1(r,ou^{-1})+iK_2(r,ou^{-1}\bigg)\sum_{|v|=n}[W_\infty]_vL(v)^{\alpha} \quad\as \label{K:proof_comparing_coefficients_full}
	\end{align}
Regarding only the imaginary part in \eqref{K:proof_comparing_coefficients_full}, we have
	\begin{align}\label{K:proof_comparing_coefficients}
		\nonumber rsY(o)&+W_\infty s^{\alpha}{K}_2(r,ou^{-1})
		\\
		&=\sum_{|v|=n}(rsL(v)[Y]_v(oU(v))+s^{\alpha}[W_\infty]_vL(v)^{\alpha}{K}_2(r,ou^{-1})\quad\as
	\end{align}
This identity holds for any $(s,u) \in \Support$. Upon dividing by $s\neq 0$, we have 
\begin{align}\label{K:proof_comparing_coefficients-2}
	\nonumber rY(o)&+W_\infty s^{\alpha-1}{K}_2(r,ou^{-1})
	\\
	&=\sum_{|v|=n}(rL(v)[Y]_v(oU(v))+s^{\alpha-1}[W_\infty]_vL(v)^{\alpha}{K}_2(r,ou^{-1})\quad\as
\end{align} By \eqref{K:assumption_m} and since $\Support$ is a group, there is $(s,u) \in \Support$ with $s>1$. If $\alpha<1$, we consider the limit of \eqref{K:proof_comparing_coefficients-2} for $s^n \to \infty$ as $n \to \infty$; for $\alpha>1$ we consider the limit of \eqref{K:proof_comparing_coefficients-2} for $s^{-n} \to 0$ as $n \to \infty$.
%
%
%
%
We obtain that
	\begin{equation}\label{K:formula_fixed_point_equation_X}
		Y(o)=\sum_{\abs{v}=n}L(v)[Y]_v(oU(v)),\quad \as
	\end{equation}
	\textit{i.e.}, identity \eqref{K:proposition_main_result_invariance_property_Y}.
Next, regarding only the real part in \eqref{K:proof_comparing_coefficients_full}, we have	
		\begin{align}
		&-\frac{1}{2}r^2s^2V(o)-s^{\alpha}W_\infty K_1(r,ou^{-1})\notag \\
		=~&-\frac{1}{2}r^2s^2\sum_{|v|=n}L(v)^2[V]_v(oU(v)) 
		-s^{\alpha}K_1(r,ou^{-1})\sum_{|v|=n}[W_\infty]_vL(v)^{\alpha} \quad\as \label{K:proof_comparing_coefficients_002}
	\end{align}
For $\alpha<2$ dividing \eqref{K:proof_comparing_coefficients_002} by $s^2$ and considering a sequence $s^n\rightarrow\infty$ immediately gives us
	\begin{equation}\label{K:formula_fixed_point_equation_Y}
		V(o)=\sum_{\abs{v}=n}L(v)^2[V](oU(v)),\quad\as
	\end{equation}
\textit{i.e.}, the identity \eqref{K:proposition_main_result_invariance_property_V}.

For $\alpha=2$ we will now prove, that $\nu^{o}$ and hence $K_1$ vanishes a.s., which in view of \eqref{K:proof_comparing_coefficients_002} implies \eqref{K:formula_fixed_point_equation_Y} as well in this case.
Recall, that for $\alpha=2$ by assumption \eqref{K:assumption_support} there is $s \neq 1$ and $u \in \OrthogonalGroup$ such that both $(1,u)$ and $(s,u)$ are in $\Support$. By group property of $\Support$, we may again assume, that $s>1$. Using the $(\Support,\alpha)$-invariance of $\bar{\nu}^o=\E\nu^o$, proved in Lemma \ref{K:proposition_levy_triplet_nu_characteristics}, we have, since $(s,u) \in \Support$,
	\begin{equation*}
			\bar{\nu}^o(sB)=s^{-\alpha} \bar{\nu}^{ou^{-1}}(B)=s^{-2} \bar{\nu}^{ou^{-1}}(B)
		\end{equation*}
	and also, since $(1,u) \in \Support$, 
	\begin{equation*}
			\bar{\nu}^o(B)=\bar{\nu}^{ou^{-1}}(B),
	\end{equation*}
	hence, by combining and iterating the above identities,
	\begin{equation}\label{eq:proof nusn B}
			\bar{\nu}^o(s^nB)=s^{-2n} \bar{\nu}^{o}(B)
		\end{equation}
	for any  $o\in\OrthogonalGroup$, $B\in\Rnot$ and $n\in\N$.
	From Lemma \ref{K:proposition_levy_triplet_nu_characteristics} we know, that $\bar{\nu}^o$ is a L\'evy measure, hence
	\begin{equation*}
			\int_{}(x^2\wedge 1)\bar{\nu}^o(dx)<\infty.
		\end{equation*}
	By direct calculations, using \eqref{eq:proof nusn B}, we have
	\begin{align*}
			\infty&>\int_{}(x^2\wedge 1)\bar{\nu}^o(dx)\ge \int_{(0,\infty)}(x^2\wedge 1)\bar{\nu}^o(\d x)=\int_{(0,1]}x^2\bar{\nu}^o(\d x)+\bar{\nu}^o((1,\infty))
			\\
			&=\sum_{n\ge 0}\int_{(s^{-n-1},s^{-n}]}x^2\bar{\nu}^o(\d x)+\bar{\nu}^o((1,\infty))=\sum_{n\ge 0}\int_{(\frac{1}{s},1]}s^{-2n}x^2\bar{\nu}^o(s^{-n}dx)+\bar{\nu}^o((1,\infty))
			\\
			&=\sum_{n\ge 0}s^{n(2-2)}\int_{(\frac{1}{s},1]}x^2\bar{\nu}^{o}(\d x)+\bar{\nu}^o((1,\infty))
			\ge \sum_{n\ge 0}s^{-2}\bar{\nu}^{o}((\tfrac1s,1])+\bar{\nu}^o((1,\infty))
			\\
			&=\infty ~s^{-2}\bar{\nu}^{o}((\tfrac1s,1])+\bar{\nu}^o((1,\infty)).
		\end{align*}
	To avoid pending contradiction, it follows necessarily, that
	\begin{equation*}
			\bar{\nu}^{o}((\tfrac1s,1])=0\quad\text{and}\quad\bar{\nu}^o((1,\infty))<\infty.
		\end{equation*}
	Note that we have required \eqref{eq:proof nusn B} and thereby the assumption \eqref{K:assumption_support} to compare the L\'evy measure of $(s^{-n-1},s^{-n}]$ to that of $(s^{-1},1]$ without changing the index $o$.
	Observe further, that
	\begin{equation*}
			\bar{\nu}^{o}((0,1])=\sum_{n\ge 0}\bar{\nu}^{o}((\tfrac{1}{s^{n+1}},\tfrac{1}{s^n}])=\sum_{n\ge 0}s^{2n}\bar{\nu}^{o}((\tfrac{1}{s},1])=0,
		\end{equation*}
	since each summand in the last sum is equal to 0. 
	Finally, by another appeal to \eqref{eq:proof nusn B}, we have that
	\begin{equation*}
			\bar{\nu}^o ((0,s^n])=\bar{\nu}^o (s^n(0,1])= s^{-2n} \bar{\nu}^o ((0,1]) = 0
		\end{equation*}
	for all $n \in \N$. Therefore, $\bar{\nu}^o((0,\infty))=0$. Using analogous arguments, it follows that $\bar{\nu}^o((-\infty,0))=0$. Hence $\nu^o$ and consequently $K_1$ vanishes a.s., and we have \eqref{K:formula_fixed_point_equation_Y} also for the case $\alpha=2$.

Finally, we define
	\begin{equation}\label{K:proof_main_result_defining_C}
			{K}(r,o):=-{K}_1(r,o)-i{K}_2(r,o)=\int_{}(1-e^{iry})\bar{\nu}^{o}(dy)
	\end{equation}
and note that $K$ inherits the property of $(\Support,\alpha)$-invariance from $K_1$ and $K_2$. 
	
Combining \eqref{K:proof_main_result_defining_C}, \eqref{K:proof_main_result_defining_Y} and \eqref{K:proof_main_result_psi_reminder} we conclude, that
	\begin{equation}\label{K:proof_main_result_general_representation_of_solution}
		\phi(roe_3)=\E\bigg[ \exp\bigg\{-W_\infty{K}(r,o)-\frac12r^2V(o)+irY(o)\bigg\}\bigg],
	\end{equation}
	with $Y$ and $V$ satisfying the a.s.\ equations \eqref{K:proposition_main_result_invariance_property_Y} and \eqref{K:proposition_main_result_invariance_property_V}, respectively.
	
	We finish this step with a remark, that $K$ vanishes a.s. in case $\alpha=2$, since $\nu^o$ does so, as has been proven above.
	
	\medskip

	\textbf{Step 2}.
	As the next step, we prove that $V$ vanishes for $\alpha <2$.
	Recall that, by Lemma \ref{K:proposition_uniform_boundedness_of_theta_o},   $V(\cdot)=\Theta^{(\cdot)}$ is almost surely bounded  by $\widetilde{C}W_\infty$ for some $\widetilde{C}>0$, uniformly in $\OrthogonalGroup$. 
	Since $\Theta^o$ is $F_\infty$-measurable and integrable (since $W_\infty$ is integrable), we have
	\begin{equation*}
		V(o)=\Theta^o=\lim_{n\rightarrow\infty}\E[\Theta^o~|~F_n] \quad \as
	\end{equation*}
	and we may use the identity \eqref{K:formula_fixed_point_equation_Y} to compute
	\begin{align}\label{K:proof_main_result_theta_o_(1)}
		\nonumber V(o)=\Theta^o&=\lim_{n\rightarrow\infty}\E[\Theta^o~|~F_n]=\lim_{n\rightarrow\infty}\E\bigg[\sum_{\abs{v}=n}L(v)^2[\Theta]^{oU(v)}_v~|~F_n\bigg]
		=\lim_{n\rightarrow\infty}\sum_{\abs{v}=n}L(v)^2\E\bigg[[\Theta]^{oU(v)}_v\bigg]
		\\
		\nonumber&\leq \widetilde{C}\lim_{n\rightarrow\infty}\sum_{\abs{v}=n}L(v)^2
		\leq \widetilde{C}\lim_{n\rightarrow\infty}\bigg(\sup_{\abs{v}=n}(L(v)^{2-\alpha})\sum_{\abs{v}=n}L(v)^{\alpha}\bigg)
		\\
		&=\widetilde{C}W_\infty\lim_{n\rightarrow\infty}\sup_{\abs{v}=n}(L(v)^{2-\alpha})=\widetilde{C}W_\infty\lim_{n\rightarrow\infty}\sup_{\abs{v}=n}(L(v)^{2-\alpha})=0,\quad\as,
	\end{align}
	where the last equality holds by relation \eqref{K:supremum_L(v)_goes_to_0} for both $\alpha\in(0,1)$ and $\alpha\in(1,2)$. Since $\Theta^o$ is non-negative by definition, this implies, that $V(o)=0$ almost surely.
	
	\medskip
	
	\textbf{Step 3}.
	In this step, we prove that $Y$ vanishes almost surely in the case $\alpha\in(0,1)$. As a consequence of Lemma \ref{K:proposition_uniform_boundedness_of_mu_o} and Lemma \ref{K:proposition_levy_triplet_integral_characteristics}, $\abs{Y(o)}=\abs{\mu^o+W_\infty{\gamma}(o)}$ is almost surely uniformly bounded in $\OrthogonalGroup$ by $\widehat{C}W_\infty$ with some constant $\widehat{C}>0$. Thus, we have
	\begin{align*}
		\abs{Y(o)}&=\lim_{n\rightarrow\infty}\bigabs{\E[Y(o)~|~F_n]}\leq\lim_{n\rightarrow\infty}\E\bigg[\biggabs{\sum_{\abs{v}=n}L(v)[Y]_v^{oU(v)}~|~F_n}\bigg]
		\\
		&\leq\lim_{n\rightarrow\infty}\sum_{\abs{v}=n}L(v)\E\bigg[\bigabs{[Y]_v^{oU(v)}}\bigg]
		\leq \widehat{C}\lim_{n\rightarrow\infty}\bigg(\sup_{\abs{v}=n}(L(v)^{1-\alpha})\sum_{\abs{v}=n}L(v)^{\alpha}\bigg)=0\quad\as
	\end{align*}
	Arguing as in the previous step, it follows, that $Y(o)=0$ almost surely whenever $\alpha \in (0,1)$,  which completes the proof.
\end{proof}
\begin{proof}[Proof of the direct inclusion in Theorem \ref{K:proposition_main_result}]
	Throughout the proof, let $r\in\R_\ge$ and $o\in\OrthogonalGroup$ be arbitrarily fixed. We rely on properties \eqref{K:proposition_main_result_invariance_property_V} and \eqref{K:proposition_main_result_invariance_property_Y} of $V$ and $Y$, as well as the disintegration property \eqref{K:martingale_W_n_limit_branching_property} of $W_\infty$ and the fact that $K$ is $(\Support,\alpha)$-invariant. Observe, that
	\begin{align*}
		&\phi(roe_3)
		=\E\bigg[\exp\bigg\{-\frac12 r^2V(o) + irY(o) - W_\infty K(r,o)\bigg\}\bigg]
		\\
		&=\E\bigg[\exp\bigg\{\sum_{\abs{v}=1}\bigg(-\frac12 r^2L(v)^2[V]_v(oU(v)) + irL(v)[Y]_v(oU(v)) - L(v)^\alpha[W_\infty]_v K(r,o)\bigg)\bigg\}\bigg]
		\\
		&=\E\bigg[\prod_{\abs{v}=1}\exp\bigg\{-\frac12 (rL(v))^2[V]_v(oU(v)) + irL(v)[Y]_v(oU(v)) - [W_\infty]_v K(rL(v),oU(v))\bigg\}\bigg].
	\end{align*}
	By conditioning the last term in the above formula on $F_1$, we immediately obtain, that it is equal to 
	\begin{equation*}
		\E\bigg[\prod_{\abs{v}=1}\phi(rL(v)oU(v)e_3)\bigg],
	\end{equation*}
	which proves the claim.
\end{proof}

\appendix

\section{Appendix}

\subsection{Geometry of $\Psi^o$}\label{sect:geometry}
In this section we investigate consequences of the embedding of random vectors with characteristic functions satisfying \eqref{K:stationary_equation} into the set of random $3 \times 3$-matrices. As we have seen already in Remark \ref{K:remark_time-dependent_solution}, this embedding allows us to interprete $\phi_t$ as the restriction of a characteristic function of $3 \times 3$-matrices; thereby giving a direct proof of continuity properties.

In the following, we want to make use of this embedding to obtain an interpretation of how the measures $\nu^o(\bf l)$, $o \in \OrthogonalGroup$, depend on $o$, by identifying them as marginals of a unique L\'evy measure  for $3 \times 3$ matrices (or vectors in $\R^9$). Unfortunately, we will obtain this identification only when fixing a realisation $\bf l$ of $\bf L$, see Remark \ref{rem:problemwithsubsequence} below for details.  Further, for simplicity, we restrict our presentation to the case $\alpha<1$ in which drift and covariance term vanish.

Let $\mathbf{l}\in H^c$, where $H^c$ is defined in Remark \ref{K:rem:null_set}. Thus the statement of Proposition \ref{prop:SimultaneousConvergence} is true for $\mathbf{l}$ and the limit $W_\infty(\mathbf{l})$ of the additive martingale defined in \eqref{K:martingale_W_n} is strictly positive. 
Let $\phi$ be a characteristic function of some random $\R^3$-vector $X=\transpose{(X_1,X_2,X_3)}$ satisfying \eqref{K:stationary_equation}. 
Let $\phi$ be a characteristic function of some random $\R^3$-vector $X=\transpose{(X_1,X_2,X_3)}$ satisfying \eqref{K:stationary_equation}. 
Now let $\widetilde{X}$ be a random ${3\times3}$-matrix, defined as
\begin{equation*}
	\widetilde{X}=\begin{pmatrix}
		0 & 0 & X_1\\
		0 & 0 & X_2\\
		0 & 0 & X_3\\
	\end{pmatrix}.
\end{equation*}
Similarly to Remark \ref{K:remark_time-dependent_solution}, we observe that
\begin{equation*}
	\phi(roe_3)=\phi_X(roe_3)=\E[ e^{r\scalar{e_3,\transpose{o}X}}]=\E [e^{\tr(\transpose{(ro)}\widetilde{X})}]=:\psi_{\widetilde{X}}(ro)
\end{equation*}
for any $r\ge 0$, where $\psi_{\widetilde{X}}$ denotes the characteristic function of $\widetilde{X}$. Let $\vecc(\widetilde{X})$ be a vectorization of $\widetilde{X}$, \textit{i.e.} the random vector $\transpose{(0,\ldots,0,X_1,X_2,X_3)}$ in $\R^9$, and note that $\tr(\transpose{o}\widetilde{X})=\scalar{\vecc(o),\vecc(\widetilde{X})}$. Therefore, we may write
\begin{equation}\label{K:formula:phi_psi_zeta}
	\phi_{X}(roe_3)\equiv \psi_{\widetilde{X}}(ro)\equiv \zeta_{\vecc(\widetilde{X})}(r\vecc(o)),
\end{equation}
where use the notation $\zeta_{\vecc(\widetilde{X})}$ for the characteristic function of $\vecc(\widetilde{X})$. 

In the remaining part of this section we reserve notations $\phi,\psi$ and $\zeta$ for characteristic functions of a random $\R^3$-vector, $\R^{3\times3}$-matrix and $\R^9$-vector, respectively. 

Recalling the definition of $M$ and using linearity and cyclic property of traces, for any $n\in\N$ and $r\ge 0$ we infer from \eqref{K:formula:phi_psi_zeta}, that
\begin{align}\label{K:geometry_Mn_equal_to_zetaYn}
	\nonumber M_n(r,o,\mathbf{l})&=\prod_{\abs{v}=n}\phi(rol(v)u(v)e_3)=\prod_{\abs{v}=n}\psi(rl(v)ou(v))
	\\
	\nonumber&=\prod_{\abs{v}=n}\E\bigg[ \exp\bigg\{irl(v)\tr(\transpose{(ou(v))}\widetilde{X}(v))\bigg\}\bigg]
	\\
	\nonumber&=\E\bigg[ \exp\bigg\{{\sum_{\abs{v}=n}ir\tr\bigg(l(v)\transpose{(ou(v))}\widetilde{X}(v)\bigg)}\bigg\}\bigg]
	\\
	\nonumber&=\E\bigg[ \exp\bigg\{{i\tr\bigg(\transpose{(ro)}\sum_{\abs{v}=n}l(v)\widetilde{X}(v)\transpose{u(v)}\bigg)}\bigg\}\bigg]
	\\
	&=\E \bigg[\exp\bigg\{{i\bigscalar{r\vecc(o),\vecc\bigg(\sum_{\abs{v}=n}l(v)\widetilde{X}(v)\transpose{u(v)}\bigg)}}\bigg\}\bigg]=:\zeta_{Y_n(\mathbf{l})}(r\vecc(o)),
\end{align}
where, as usual, $(\widetilde{X}(v))_{v}$ denote independent copies of $\widetilde{X}$.
From \eqref{K:geometry_Mn_equal_to_zetaYn} we infer, that $M_n(r,o,\mathbf{l})$ can be considered as the {\em restriction} of the characteristic function $\zeta_{Y_n(\mathbf{l})}$ of a nine-dimensional random vector \begin{equation}\label{K:geometry_def_Yn}
	Y_n(\mathbf{l}):=\sum_{\abs{v}=n}l(v)\vecc(\widetilde{X}(v)\transpose{u(v)}),
\end{equation} 
evaluated {\em only} at $r\vecc(o)$, for any $n\in\N$ .  In view of \eqref{K:supremum_L(v)_goes_to_0}, $(Y_n(\mathbf{l}))_{n\in\N}$ are row sums in a triangular null array.

However, in contrast to the (one-dimensional) considerations in Proposition \ref{prop:SimultaneousConvergence}, we cannot show that $\zeta_{Y_n(\mathbf{l})}$ converges as a function on $\R^9$ as $n \to \infty$, because the convergence of the multiplicative martingale $M_n(r,o,\mathbf{l})$ only yields the pointwise convergence of $\zeta_{Y_n(\mathbf{l})}$ at the lower-dimensional subset of points $\{r\vecc(o):~r\in\R_\ge,o\in\OrthogonalGroup\}$. That is, we cannot conclude that the row sums  $(Y_n(\mathbf{l}))_{n\in\N}$ in the triangular array converge in distribution.

What we can do instead is to prove, for fixed $\mathbf{l}$, that the sequence $(Y_n(\mathbf{l}))_{n\in\N}$ has a convergent subsequence, the limit of which can be identified as an infinitely divisible random vector by means of \cite[Theorem 13.28]{Kallenberg:1997}.

\begin{lem}\label{K:prop:construction_of_Yl}
	For any $\mathbf{l}\in N^c$ let $(Y_n(\mathbf{l}))_n$ be a sequence of $\R^9$-vectors defined by \eqref{K:geometry_def_Yn}. Suppose, that assumptions \eqref{K:assumption_m},\eqref{K:assumption_derivative_m} and \eqref{K:assumption_regular_variation} hold with $\alpha\in(0,1)$. Then there exists an infinitely divisible random vector $Y(\mathbf{l})$ in $\R^9$ and a subsequence $(Y_{n_m}(\mathbf{l}))_m$, such that 
	\begin{equation}\label{K:prop:Ynk_to_Y}
		Y_{n_m}(\mathbf{l})\dc Y(\mathbf{l}), \quad m\rightarrow\infty.
	\end{equation}
	Moreover, for any $r\ge 0$ and $o\in\OrthogonalGroup$ one has
	\begin{equation}\label{K:prop:Minfty_equal_zetaY}
		M_\infty(r,o,\mathbf{l})=\zeta_{Y(\mathbf{l})}(r\vecc(o)),
	\end{equation}
	where $\zeta_{Y(\mathbf{l})}$ is the characteristic function of $Y(\mathbf{l})$ and $M_\infty$ is the a.s. limit of the multiplicative martingale $M$, defined by formula \eqref{K:multiplicative_martingale_M}.
\end{lem}

\begin{rem}\label{rem:problemwithsubsequence}
	Note that the choice of the subsequence $(Y_{n_m}(\mathbf{l}))_m$ will in general depend on $\mathbf{l}$ and we cannot use  relation \eqref{K:prop:Minfty_equal_zetaY} to prove that all subsequential limits are equal, since \eqref{K:prop:Minfty_equal_zetaY} does not evaluate $\zeta_{Y(\bf l)}$ at a dense set of points. We have tried hard to consider $\zeta_{Y(\bf l)}$ as a characteristic function not on 9-dimensional space, but on some suitable lower-dimensional set of equivalence classes, but have failed so far.
	
	This has made it impossible for us to prove the measurability of $Y$ or $\zeta_Y$ as a function of $\mathbf{l}$, which would be required to obtain a representation of $\zeta_{\vec{\tilde{X}}}$ in terms of random L\'evy-Khintchine exponent as in \eqref{K:characteristic_exponent_Psi}.
\end{rem}

\begin{proof}[Proof of Lemma \ref{K:prop:construction_of_Yl}] Our approach is based on direct construction of the limiting infinitely divisible random vector $Y(\mathbf{l})$, which satisfies our needs. 
	In steps 1--3 we construct the three components of the (deterministic) L\'evy-triplet, associated with $Y(\mathbf{l})$, namely the L\'evy measure $\nu(\mathbf{l})$, shift $\mu(\mathbf{l})$ and covariance matrix $\Theta(\mathbf{l})$. In step 4 we prove the assertions \eqref{K:prop:Ynk_to_Y} and \eqref{K:prop:Minfty_equal_zetaY}. 
	
	Throughout the proof we assume $\mathbf{l}\in N^c$ to be fixed. To simplify notations, we denote $\xi(v):=l(v)\vecc(\widetilde{X}(v)\transpose{u(v)})$, hence
	\begin{equation*}
		Y_n(\mathbf{l})=\sum_{\abs{v}=n}\xi(v).
	\end{equation*}
	
	\textbf{Step 1}: 
	We start with the  construction of the L\'evy measure. Consider a family of (deterministic) measures $(\nu_n(\mathbf{l}))_n$ on $\R^9$, defined by 
	\begin{equation*}
		\nu_n(\mathbf{l}):=\sum_{\abs{v}=n}\P\circ \xi(v)^{-1}\quad\text{on $\R^9\setminus\{0\}$}
	\end{equation*}
	and $\nu_n(\mathbf{l})(\{0\}):=0$.
	We prove, that there exists a subsequence of $(\nu_n(\mathbf{l}))_n$ which converges vaguely on $\R^9\setminus \{0\}$, and its limit, which we denote by $\nu(\mathbf{l})$, is a L\'evy measure. By assumption \eqref{K:assumption_regular_variation}, for any $r>0$ we have
	\begin{align}\label{K:proof:bounded_nu_n}
		\nonumber\nu_n(\mathbf{l})(B^c_r(0))&=\sum_{\abs{v}=n}\P\bigg(\xi(v)\in B^c_r(0)\bigg)
		=\sum_{\abs{v}=n}\P\bigg(\norm{l(v)\vecc(\widetilde{X}(v)\transpose{u(v)})}\ge r\bigg)
		\\
		\nonumber&=\sum_{\abs{v}=n}\P\bigg(l(v)\norm{\widetilde{X}(v)\transpose{u(v)}}_{\mathbf{F}}\ge r\bigg)
		=\sum_{\abs{v}=n}\P\bigg(l(v)\norm{\widetilde{X}(v)}_{\mathbf{F}}\ge r\bigg)
		\\
		&=\sum_{\abs{v}=n}\P\bigg(l(v)\norm{X(v)}\ge r\bigg)
		\leq \sum_{\abs{v}=n} Cr^{-\alpha}l(v)^\alpha = Cr^{-\alpha} W_n(\mathbf{l}),
	\end{align}
	were $\norm{\cdot}_{\mathbf{F}}$ is the Frobenius norm and $W_n(\mathbf{l})$ is a deterministic constant, which is the realization of an additive martingale $(W_n)_n$. Recall, that on $N^c$ we have $W_n(\mathbf{l})\rightarrow W_\infty(\mathbf{l})\in(0,\infty)$, and hence $\sup_n W_n(\mathbf{l})<\infty$. Fix any function $f\in C^+_c(\R^9\setminus \{0\})$ and assume w.l.o.g. that $f$ is supported on a compact set $A\subset B^c_\delta(0)\subset \R^9\setminus\{0\}$ for some $\delta>0$. By \eqref{K:proof:bounded_nu_n} we have
	\begin{equation*}
		\sup_{n\in\N}\int f \d\nu_n(\mathbf{l})\leq \sup_{n\in\N}\bigg\{\nu_n(\mathbf{l})(B_\delta^c(0))\max_{x\in A}f(x)\bigg\}\leq\widetilde{C}\delta^{-\alpha}<\infty
	\end{equation*}
	for some constant $\widetilde{C}$. Thus we may apply the converse implication in \cite[Theorem A2.3]{Kallenberg:1997}, part (ii), and conclude, that there exists a subsequence $(\nu_{n_k}(\mathbf{l}))_k$, such that 
	\begin{equation}\label{K:proof:convergence_nuhk}
		{\nu_{n_k}(\mathbf{l})}\overset{v}{\rightarrow} \nu(\mathbf{l}),\quad k\rightarrow\infty,
	\end{equation} 
	for some locally finite measure $\nu(\mathbf{l})$ on $\R^9\setminus\{0\}$, where $\overset{v}{\rightarrow}$ denotes vague convergence.
	Since also $\nu_n(\mathbf{l})(\{0\})=0$ for all $n$, we may well-define $\nu(\mathbf{l})(\{0\}):=0$. The fact, that $\nu(\mathbf{l})$ is a L\'evy measure then follows from the estimate
	\begin{align*}
		\int (\norm{x}^2 \wedge 1)\nu(\mathbf{l})(\d x)\leq \nu(B^c_1(0))+&\int_{B_1(0)}\norm{x}^2 \nu(\mathbf{l})(\d x)\leq \widetilde{C}_1+\int_{(0,1]}t^2\d \nu(\mathbf{l})(B^c_t)
		\\
		&\leq\widetilde{C}_1+\int_0^1 t^2 \d \widetilde{C}_2 t^{-\alpha}=\widetilde{C}_1 + \widetilde{C}_2\int_0^1 t^{1-\alpha}\d t < \infty
	\end{align*} 
	for some constants $\widetilde{C}_1$ and $\widetilde{C}_2$, where the upper bound on $\nu(\mathbf{l})(B_t^c)$ is derived from \eqref{K:proof:bounded_nu_n}. 
	We may therefore assume, that $\nu(\mathbf{l})$ is the deterministic L\'evy measure which corresponds to an infinitely divisible random $\R^9$-vector $Y(\mathbf{l})$. 
	
	\textbf{Step 2}: 
	In the next step we construct the shift $\mu(\mathbf{l})$, associated with $Y(\mathbf{l})$.
	We fix any $h\in(0,1)$ such that $\nu(\mathbf{l})(\{\norm{x}=h\})=0$. Define a sequence of nine-dimensional vectors $(\mu_n^h(\mathbf{l}))_n$ as
	\begin{equation*}
		\mu_n^h(\mathbf{l}):=\sum_{\abs{v}=n} \E [\xi(v):~\norm{\xi(v)}\le h].
	\end{equation*}
	Utilizing assumption  \eqref{K:assumption_regular_variation}, we observe that
	\begin{align*}
		\norm{\mu_n^h(\mathbf{l})}&=\bignorm{\sum_{\abs{v}=n} \E [\xi(v):~\norm{\xi(v)}\le h]}
		\leq \sum_{\abs{v}=n} \E \bigg[\norm{\xi(v)} \1_{\norm{\xi(v)}\le h}\bigg]
		\\
		&\leq \sum_{\abs{v}=n} \int_0^h \P\bigg(\norm{\xi(v)} > x\bigg) \d x
		\leq \sum_{\abs{v}=n} \int_0^h \P\bigg(\norm{l(v)\vecc(\widetilde{X}(v)\transpose{u(v)})} > x\bigg) \d x
		\\
		&= \sum_{\abs{v}=n} \int_0^h \P\bigg(l(v)\norm{X(v)} > x\bigg) \d x\leq \sum_{\abs{v}=n}C l(v)^\alpha \int_0^h x^{-\alpha} \d x < \widetilde{C}\sup_n W_n(\mathbf{l})<\infty
	\end{align*}
	for some constant $\widetilde{C}$. This implies, that the sequence $(\norm{\mu_n^h(\mathbf{l})})_n$ is uniformly bounded. Furthermore, for any $n\in\N$ each of nine entries of the vector  $\mu_n^h(\mathbf{l})=\transpose{(\{\mu_n^h(\mathbf{l})\}_1,\ldots,\{\mu_n^h(\mathbf{l})\}_9)}$ is bounded by $\norm{\mu_n^h(\mathbf{l})}$. This implies uniform boundedness in $\R^9$ of the sequence $(\mu_n^h(\mathbf{l}))_n$, hence there exists a convergent subsequence $(\mu_{n_j}^h(\mathbf{l}))_j$ such that
	\begin{equation}\label{K:proof:convergence_muhj} 
		{\mu_{n_j}^h(\mathbf{l})}\rightarrow \mu^h(\mathbf{l}), \quad j\rightarrow\infty,
	\end{equation} 
	for a certain finite vector $\mu^h(\mathbf{l})$ in $\R^9$. Let us define
	\begin{equation}\label{K:proof:def_mul}
		\mu(\mathbf{l}):=\mu^h(\mathbf{l}) + \int_ {\{h<\norm{x}\le 1\}} \nu(\mathbf{l})(\d x).
	\end{equation}
	Since the integral in the right-hand side is finite as the consequence of \eqref{K:proof:bounded_nu_n}, $\mu(\mathbf{l})$ is  finite as well. We may therefore assume, that $\mu(\mathbf{l})$ is a (deterministic) shift in the L\'evy-triplet, associated with $Y(\mathbf{l})$.

	\textbf{Step 3}: 
	It remains to construct the covariance matrix $\Theta(\mathbf{l})$. We proceed in a similar way as before. Let $h$ be the same, as the one used in the definition of $(\mu_n^h)_n$ previously in step 2. Define a sequence of matrices $(\Theta_n^h(\mathbf{l}))_n$ in $\R^{9\times 9}$ as
	\begin{equation*}
		\Theta_n^h(\mathbf{l}):=\sum_{\abs{v}=n} \Cov \bigg[\xi(v):~\norm{\xi(v)}\le h\bigg].
	\end{equation*}
	We proceed with proving the boundedness of entries on the main diagonal of $\Theta_n^h(\mathbf{l})$. Denote by $\{\Theta_n^h(\mathbf{l})\}_{ij}$ an entry in the $i$-th row and $j$-th column of $\Theta_n^h(\mathbf{l})$, $i,j=1,\ldots,9$. Likewise, let $\{\xi(v)\}_i$ be the $i$-th entry of $\xi(v)$. Direct calculations based on the estimate in assumption \eqref{K:assumption_regular_variation} yield
	\begin{align}\label{K:proof:Vxi_i}
		\nonumber\V&\bigg[\{\xi(v)\}_i:~\norm{\xi(v)}\le h\bigg]
		\le \E\bigg[\{\xi(v)\}_i\1_{\norm{\xi(v)}\le h}\bigg]^2 - \bigg(\E\bigg[\{\xi(v)\}_i\1_{\norm{\xi(v)}\le h}\bigg]\bigg)^2
		\\
		\nonumber&\le \E\bigg[\{\xi(v)\}_i\1_{\norm{\xi(v)}\le h}\bigg]^2 
		\le \E\bigg[(\{\xi(v)\}_i)^2\1_{(\{\xi(v)\}_i)^2\le h}\bigg]
		= \int_{0}^{h}\P\bigg((\{\xi(v)\}_i)^2 >x \bigg) \d x
		\\
		&\le \int_{0}^{h}\P\bigg(\norm{\xi(v)}^2 > x\bigg) \d x  \le \int_{0}^{h}\P\bigg(l(v)\norm{X} > \sqrt{x}\bigg) \d x \le C l(v)^\alpha \int_{0}^{h}x^{-\frac{\alpha}{2}}\d x\leq \widetilde{C} l(v)^\alpha
	\end{align}
	for some constant $\widetilde{C}<\infty$. Therefore, for any $i=1,\ldots,9$ we obtain
	\begin{equation}\label{K:proof:sum_Vxi_i}
		[\Theta_n^h(\mathbf{l})]_{ii}=\sum_{\abs{v}=n} \V\bigg[\{\xi(v)\}_i:~\norm{\xi(v)}\le h\bigg] \leq \widetilde{C}\sup_n W_n(\mathbf{l}).
	\end{equation}
	
	Furthermore, by \eqref{K:proof:Vxi_i} all off-diagonal elements $\{\Theta_n^h(\mathbf{l})\}_{ij}$ can be bounded as
	\begin{align}\label{K:proof:sum_Vxi_ij}
		\nonumber\{\Theta_n^h(\mathbf{l})\}_{ij}\le \sum_{\abs{v}=n}\bigg(\V\bigg[\{\xi(v)\}_i:~\norm{\xi(v)}\le h\bigg] \V\bigg[\{\xi(v)\}_j&:~\norm{\xi(v)}\le h\bigg]\bigg)^\frac12
		\\
		&\le \widetilde{C}\sup_n W_n(\mathbf{l}) < \infty,
	\end{align}
	where we used the fact, that $\Cov[X,Y]\leq \sqrt{\V(X) \V(Y)}$ is true for any $X$ and $Y$. Combining \eqref{K:proof:sum_Vxi_i} and \eqref{K:proof:sum_Vxi_ij} we infer, that
	$\norm{\Theta_n^h(\mathbf{l})}_{\mathbf{F}} \le f(\widetilde{C}\sup_n W_n(\mathbf{l}))$ for a certain continuous function $f$. Since $\sup_n W_n(\mathbf{l})$ is finite under our assumptions, it follows that the sequence $(\Theta_n^h(\mathbf{l}))_n$ is uniformly bounded in $\R^{9\times 9}$. Hence, there exists a subsequence $(\Theta_{n_i}^h(\mathbf{l}))_{i}$ such that
	\begin{equation}\label{K:proof:convergence_thetahi} 
		{\Theta_{n_i}^h(\mathbf{l})}\rightarrow \Theta^h(\mathbf{l}), \quad i\rightarrow\infty,
	\end{equation} 
	where $\Theta^h(\mathbf{l})$ is a bounded (in norm) matrix in $\R^{9\times 9}$. Since $\nu(\mathbf{l})$ is a L\'evy measure, a matrix defined as
	\begin{equation}\label{K:proof:def_thetal}
		\Theta(\mathbf{l}):=\Theta^h(\mathbf{l}) + \int_{\norm{x}\le h}x\transpose{x}\nu(\mathbf{l})(\d x)
	\end{equation}
	is finite. Since the matrix ${\Theta_{n_i}^h(\mathbf{l})}$ is symmetric as the sum of symmetric matrices for any $i\in\N$,  the limit $\Theta^h(\mathbf{l})$ and consequently $\Theta(\mathbf{l})$ are both symmetric matrices as well. By similar arguments, $\Theta(\mathbf{l})$ is also positive semi-definite. It follows, that $\Theta(\mathbf{l})$ is a covariance matrix of some vector in $\R^9$. As in step 1 and 2, we assume that $\Theta(\mathbf{l})$ is a (deterministic) covariance matrix, associated with $Y(\mathbf{l})$.
	
	\textbf{Step 4}: It remains to prove, that there exists a subsequence of $(Y_n(\mathbf{l}))_n$, which converges in distribution to $Y(\mathbf{l})$. Combining results \eqref{K:proof:convergence_nuhk}, \eqref{K:proof:convergence_muhj} and \eqref{K:proof:convergence_thetahi}, by a diagonal argument we infer, that there exists a sequence of indices $(n_m)_m\subset \N$, such that for any $h\in(0,1)$ with $\nu(\mathbf{l})(\{\norm{x}=h\})=0$ the relations
	\begin{align*}
		&\sum_{\abs{v}=n_m}\P\circ \bigg(l(v)\vecc(\widetilde{X}(v)\transpose{u(v)})\bigg)^{-1}\overset{v}{\rightarrow}\nu(\mathbf{l})\quad\text{on $\R^9\setminus\{0\}$};
		\\
		&\sum_{\abs{v}=n_m}\E \bigg[l(v)\vecc(\widetilde{X}(v)\transpose{u(v)}):~\bignorm{l(v)\vecc(\widetilde{X}(v)\transpose{u(v)})}\le h\bigg] \rightarrow \mu^h(\mathbf{l});
		\\
		&\sum_{\abs{v}=n_m}\Cov \bigg[l(v)\vecc(\widetilde{X}(v)\transpose{u(v)}):~\bignorm{l(v)\vecc(\widetilde{X}(v)\transpose{u(v)})}\le h\bigg] \rightarrow \Theta^h(\mathbf{l})
	\end{align*}
	hold simultaneosly, as $m\rightarrow\infty$. By the converse implication in \cite[Theorem 13.28]{Kallenberg:1997} we infer, that 
	\begin{equation*}
		Y_{n_m}(\mathbf{l})=\sum_{\abs{v}=n_m} l(v)\vecc(\widetilde{X}(v)\transpose{u(v)})
	\end{equation*}
	converges in distribution to a limiting infinitely divisible $\R^9$-vector, as $m\rightarrow\infty$, which in view of \eqref{K:proof:def_mul} and \eqref{K:proof:def_thetal} coincides with $Y(\mathbf{l})$ and has the L\'evy-triplet $(\mu(\mathbf{l}),\Theta(\mathbf{l}),\nu(\mathbf{l}))$ corresponding to it. Thus, the convergence in \eqref{K:prop:Ynk_to_Y} is proved. Applying formula \eqref{K:geometry_Mn_equal_to_zetaYn}, for any $r\in\R_\ge$ and $o\in\OrthogonalGroup$ we have
	\begin{equation*}
		M_{n_m}(r,o,\mathbf{l})=\zeta_{Y_{n_m}(\mathbf{l})}(r\vecc(o)).
	\end{equation*}
	Since by \eqref{K:prop:Ynk_to_Y} we particularly have the convergence of characteristic functions,
	by sending $m\rightarrow\infty$ in the above formula we obtain
	\begin{equation*}
		M_\infty(r,o,\mathbf{l})=\zeta_{Y(\mathbf{l})}(r\vecc(o)).
	\end{equation*}
	The proof is complete.
\end{proof}

By comparing the represenations of $M_\infty(r, o, \bf l)$ derived in Lemma  \ref{K:prop:construction_of_Yl} and Proposition \ref{K:proposition_solution_via_multiplicative_martingale_M}, respectively, we obtain an interpretation of $\nu^o$, $\mu^o$ and $\Theta^o$ as marginals of the corresponding objects in the L\'evy-Khintchine exponent of $\zeta_{Y(\bf l)}$.

\begin{prop}\label{K:prop:geometry_of_Psio}
	Assume that \eqref{K:assumption_m}, \eqref{K:assumption_derivative_m} and \eqref{K:assumption_regular_variation} hold with $\alpha\in(0,1)$. Fix ${\bf l }\in N^c$ and let $Y(\bf l)$ be an infinitely divisible random vector in $\R^9$ with the corresponding L\'evy-triplet $(\mu({\bf l}),\Theta({\bf l}),\nu({\bf l}))$ given by Lemma \ref{K:prop:construction_of_Yl}, and denote by $\Psi({\bf l})$ the characteristic exponent of $Y({\bf l})$. For any fixed $o\in\OrthogonalGroup$ let $\Psi^o$ and $(\mu^o,\Theta^o,\nu^o)$ be the characteristic exponent and the L\'evy-triplet respectively, given by Proposition \ref{K:proposition_solution_via_multiplicative_martingale_M}. Then, the following identities hold almost surely:
	\begin{flalign}
		\nonumber \textit{(i)}\quad &\mu^o({\bf l}) = \scalar{\vecc(o),\mu({\bf l})} - \int_{} \scalar{\vecc(o),y} 	\bigg(\1_{[0,1]}(\norm{y}) - \1_{[0,1]}(\abs{\scalar{\vecc(o),y}}) \bigg) \nu({\bf l})(\d y);
		&\\
		\nonumber \textit{(ii)}\quad &\Theta^o({\bf l}) = \transpose{\vecc(o)}\Theta({\bf l}) \vecc(o);
		&\\
		\textit{(iii)}\quad &\int f(y){\nu}^o({\bf l})(\d y)=\int_{} f(\scalar{\vecc(o),y})\nu({\bf l})(\d y)~\text{for 	all $f\in C_c(\R\setminus \{0\})$}.
	\end{flalign}
\end{prop}

\begin{proof}
	This follows by standard calculations when comparing the L\'evy-Khintchine exponent $\Psi^o(\bf l)$ with that of $\scalar{ \vec{o},Y(\bf l)}$ c.f. e.g \cite[1.3.3 (b)]{Hazod2001}.
\end{proof}

\subsection*{Acknowledgement}
The authors are very grateful to Alexander Marynych for various fruitful discussions on the subject.

\end{document}